\documentclass[reqno]{amsart}
\usepackage{amsmath, amsthm, amssymb, amstext}

\usepackage{hyperref,xcolor}
\hypersetup{
  pdfborder={0 0 0},
  colorlinks,
}
\usepackage{enumitem}
\newtheorem{theorem}{Theorem}
\newtheorem{remark}[theorem]{Remark}
\newtheorem{lemma}[theorem]{Lemma}
\newtheorem{proposition}[theorem]{Proposition}
\newtheorem{corollary}[theorem]{Corollary}

\DeclareMathOperator*{\divergenz}{div}              %
\DeclareMathOperator*{\ints}{int}         %
\DeclareMathOperator*{\loc}{loc}         %
\DeclareMathOperator*{\supp}{supp}         %
\DeclareMathOperator*{\essinf}{ess\,inf}         %
\DeclareMathOperator*{\esssup}{ess\,sup}         %

\def\N{\mathbb{N}}

\def\R{\mathbb{R}}
\def\C{\mathbb{C}}
\def\RN{\mathbb{R}^N}

\def\W1p{W^{1,p}(\Omega)}
\def\Wx1p{W^{1,p(\cdot)}(\Omega)}

\def\eps{\varepsilon}
\def\ph{\varphi}
\def\Om{\Omega}
\def\rand{\partial \Omega}

\def\into{\int_{\Omega}}

\def\C1{\ints (C^1(\overline{\Om})_+)}

\def\close{\overline{\Omega}}

\numberwithin{theorem}{section}
\numberwithin{equation}{section}

\def\cprime{$'$}

\title[Global a priori bounds to quasilinear parabolic equations]{Global a priori bounds for weak solutions to quasilinear parabolic equations with nonstandard growth}
\author[P. Winkert]{Patrick Winkert}
\address{Technische Universit\"{a}t Berlin, Institut f\"{u}r Mathematik,\\ Stra\ss e des 17.\,Juni 136, 10623 Berlin, Germany}
\email{winkert@math.tu-berlin.de}
\author[R. Zacher]{Rico Zacher}
\address{Universit\"at Ulm, Institut f\"ur Angewandte Analysis, Helmholtzstra\ss e 18, 89069 Ulm, Germany}
\email{rico.zacher@uni-ulm.de}
\begin{document}

\begin{abstract}
In this paper we study a rather wide class of quasilinear parabolic problems with nonlinear boundary condition and
nonstandard growth terms. It includes the important case of equations with a $p(t,x)$-Laplacian. By means of the
localization method and De Giorgi's iteration technique we derive global a priori bounds for weak solutions of such problems.
Our results seem to be new even in the constant exponent case. 
\end{abstract}
\maketitle

\section{Introduction}
This paper is concerned with a rather wide class of quasilinear parabolic problems with nonlinear boundary condition.
An important feature of the problems under study is that they may contain nonlinear terms with variable
growth exponents depending on time and space. To be more precise, 
let $\Omega \subset \R^N, N>1,$ be a bounded domain with Lipschitz boundary $\Gamma:=\rand$ and let $T>0, Q_T= (0,T) \times \Omega$ and $\Gamma_T=(0,T) \times \Gamma $. Given $p \in C(\overline{Q}_T)$ satisfying $1 < p^-=\inf_{(t,x) \in \overline{Q}_T} p(t,x)$, the main purpose of the paper consists in proving global a priori bounds for weak solutions of parabolic equations of the form
\begin{equation}\label{problem}
    \begin{aligned}
	u_t - \divergenz \mathcal{A}(t,x,u,\nabla u)  & = \mathcal{B}(t,x,u,\nabla u) \quad && \text{in } Q_T,\\
	\mathcal{A}(t,x,u,\nabla u) \cdot \nu & = \mathcal{C}(t,x,u)  &&\text{on } \Gamma_T,\\
        u(0,x)& =u_0(x) && \text{in } \Omega.
    \end{aligned}
\end{equation}
Here $\nu(x)$ denotes the outer unit normal of $\Omega$ at $x \in \Gamma$, $u_0 \in L^2(\Omega)$ and the nonlinearities involved $\mathcal{A}: Q_T \times \R \times \RN \to \RN$, $\mathcal{B}: Q_T \times \R \times \RN \to \R$ and $\mathcal{C}: \Gamma_T \times \R \to \R$ are assumed to satisfy appropriate $p(t,x)$-structure conditions which are stated in hypothesis (H), see below. Our setting includes as a special case parabolic equations with a $p(t,x)$-Laplacian, which is given by 
\begin{align*}
    \Delta_{p(t,x)}u=\divergenz \left(\left|\nabla u \right|^{p(t,x)-2}\nabla u  \right),
\end{align*}
and which reduces to the $p(x)$-Laplacian if $p(t,x)=p(x)$, respectively, to the well-known $p$-Laplacian in case $p(t,x) \equiv p$.

Nonlinear equations of the type considered in \eqref{problem} with variable exponents in the structure conditions are usually termed equations with nonstandard growth. Such equations are of great interest and occur in the mathematical modelling of certain physical phenomena, for example in fluid dynamics (flows of electro-rheological fluids or fluids with temperature-dependent viscosity), in nonlinear viscoelasticity, in image processing and in processes of filtration through porous media, see for example, Antontsev-D{\'{\i}}az-Shmarev \cite{Antontsev-Diaz-Shmarev-2002}, Antontsev-Rodrigues \cite{Antontsev-Rodrigues-2006}, Chen-Levine-Rao \cite{Chen-Levine-Rao-2006}, Diening \cite{Diening-2002}, Rajagopal-R\r u\v zi\v cka \cite{Rajagopal-Ruzicka-2001}, R\r u\v zi\v cka \cite{Ruzicka-2000} and Zhikov \cite{Zhikov-1997}, \cite{Zhikov-1997-b} and the references therein.

Throughout the paper we impose the following conditions.
\begin{enumerate}[leftmargin=0.8cm]
    \item[(H)]
	The functions $\mathcal{A}: Q_T \times \R \times \RN \to \RN$, $\mathcal{B}: Q_T \times \R \times \RN \to \R$ and
	$\mathcal{C}: \Gamma_T \times \R \to \R$ are Carath\'eodory functions satisfying the subsequent structure conditions:
	\begin{align*}
	    \text{(H1) } & |\mathcal{A}(t,x,s,\xi)| \leq a_0|\xi|^{p(t,x)-1}+a_1|s|^{q_1(t,x)\frac{p(t,x)-1}{p(t,x)}}+a_2, && \text{ a.e. in } Q_T,\\
	    \text{(H2) } & \mathcal{A}(t,x,s,\xi) \cdot \xi \geq a_3|\xi|^{p(t,x)}-a_4|s|^{q_1(t,x)}-a_5, && \text{ a.e. in } Q_T,\\
	    \text{(H3) } & |\mathcal{B}(t,x,s,\xi)| \leq b_0|\xi|^{p(t,x)\frac{q_1(t,x)-1}{q_1(t,x)}}+b_1|s|^{q_1(t,x)-1}+b_2, &&\text{ a.e. in } Q_T,\\
	    \text{(H4) } & |\mathcal{C}(t,x,s)| \leq c_0|s|^{q_2(t,x)-1}+c_1, && \text{ a.e. in } \Gamma_T,
	\end{align*}
	for all $s \in \R$, all $\xi \in \RN$ and with positive constants $a_i, b_j, c_l$.
	Further, $p\in C(\overline{Q}_T)$ with $\inf_{(t,x) \in \overline{Q}_T} p(t,x)>1$ and $q_1 \in C(\overline{Q}_T)$ as well as $q_2\in C(\overline{\Gamma}_T)$ are chosen such that
	\begin{align*}
	    & p(t,x) \leq q_1(t,x)<p^*(t,x), \quad (t,x)\in \overline{Q}_T,\\
	    & p(t,x)\leq q_2(t,x)<p_*(t,x),\quad (t,x)\in \overline{\Gamma}_T,
	\end{align*}
	with the critical exponents
	\begin{align*}
	    p^*(t,x)= p(t,x)\frac{N+2}{N}, \qquad p_*(t,x)= p(t,x)\frac{N+2}{N}-\frac{2}{N}.
	\end{align*}
\end{enumerate}

\begin{enumerate}[leftmargin=0.8cm]
    \item[(P)]
        The exponent $p\in C(\overline{Q}_T)$ is log-H\"{o}lder continuous on $Q_T$, that is, there exists $k>0$ such that
        \begin{align*}
            |p(t,x)-p(t',x')| \leq \frac{k}{\log \left(e+\displaystyle\frac{1}{|t-t'|+|x-x'|} \right)},
        \end{align*}
        for all $(t,x), (t',x') \in Q_T$.
\end{enumerate}

A function $u: Q_T \to \R$ is called a {\bf weak solution} ({\bf subsolution, supersolution}) of problem (\ref{problem}) if
\begin{align*}
    u \in \mathcal{W}:=\left\{v\in C\left([0,T];L^2(\Omega)\right): |\nabla v| \in L^{p(\cdot,\cdot)}(Q_T) \right\}
\end{align*}
such that
\begin{align}\label{weak_solution}
    \begin{split}
        -&\into u_0 \ph dx\bigg |_{t=0} - \int_{0}^{T}\into u \ph_t dx dt+ \int_{0}^{T}\into \mathcal{A}(t,x,u,\nabla u) \cdot \nabla \ph dxdt\\
	& =\,(\le,\,\ge) \int_{0}^{T} \into \mathcal{B}(t,x,u,\nabla u) \ph dxdt  + \int_{0}^{T} \int_{\Gamma} \mathcal{C}(t,x,u) \ph d \sigma dt
    \end{split}
\end{align}
holds for all nonnegative test functions
\begin{align*}
    \ph \in \mathcal{V}:=\left\{\psi \in W^{1,2}\left([0,T];L^2(\Omega)\right): |\nabla \psi| \in L^{p(\cdot,\cdot)}(Q_T) \right\},
\end{align*}
with $\varphi \big |_{t=T}=0$, where $d \sigma$ denotes the $(N-1)$-dimensional surface measure.

Using the notation $y_+=\max(y,0)$, our main result reads as follows.

\begin{theorem}\label{maintheorem}
    Let the assumptions in (H) and (P) be satisfied. Then there exist
    positive constants $\alpha=\alpha(T)$, $\beta=\beta(p,q_1,q_2)$ and
    \begin{align*}
	C=C(p,q_1,q_2,a_3,a_4,a_5,b_0,b_1,b_2,c_0,c_1,N,\Omega,T)
    \end{align*}
   such that the following assertions hold.
    \begin{enumerate}[leftmargin=0.7cm]
        \item[(A)]
            If $u \in \mathcal{W}$ is a weak subsolution of (\ref{problem}) and if $u_0 \in L^2(\Omega)$ is essentially bounded above in $\Omega$, then both $\esssup_{(0,T) \times \Om} u$ and $\esssup_{(0,T) \times \Gamma} u$ are bounded
            from above by
            \begin{align*}
                &  2^{\alpha} \max\left(\esssup_{\Omega}u_0 , C \left [1+\int_0^T\!\into u_+^{q_1(t,x)} dxdt + \int_0^T\!\int_{\Gamma} u_+^{q_2(t,x)} d \sigma dt\right]^{\beta}\right).
            \end{align*}
        \item[(B)]
            If $u \in \mathcal{W}$ is a weak supersolution of (\ref{problem}) and if $u_0 \in L^2(\Omega)$ is essentially bounded below in $\Omega$, then both $\essinf_{(0,T) \times \Om} u$ and $\essinf_{(0,T) \times \Gamma} u$ are bounded from
            below by
            \begin{align*}
                & -2^\alpha \max\left(-\essinf_{\Omega}u_0, C
                \left [1+\int_0^T\!\!\into (-u)_+^{q_1(t,x)}  dxdt + \int_0^T\!\! \int_{\Gamma} (-u)_+^{q_2(t,x)}   d \sigma dt \right ]^\beta\right).
            \end{align*}
    \end{enumerate}
\end{theorem}

Note that the assumptions of Theorem \ref{maintheorem} imply that the bounds given in Part (A) and (B) are finite.
In fact, for $u\in \mathcal{W}$  the finiteness of the integral terms in (A) and (B) can be seen by means of localization
($p$ is continuous) and the parabolic embeddings from Proposition \ref{energy_estimate} below.

Since a weak solution of (\ref{problem}) is both, a weak subsolution and a weak supersolution of (\ref{problem}), an important consequence of Theorem \ref{maintheorem} is stated in the following corollary.

\begin{corollary}\label{corollary}
    Let the assumptions (H) and (P) be satisfied and let $u_0 \in L^\infty(\Omega)$. Then, every weak solution $u \in \mathcal{W}$ of (\ref{problem}) is essentially bounded both in $(0,T) \times \Omega$ and on $(0,T) \times \Gamma$
    (the latter w.r.t.\,the surface measure on $\Gamma$),  and the estimates in (A) and (B) from Theorem \ref{maintheorem} give a lower and an upper bound of $u$ on $(0,T) \times \Omega$ and $(0,T) \times \Gamma$, respectively.
\end{corollary}

In case that $p$ does not depend on $t$, the following result is valid.

\begin{theorem}
    If the exponent $p$ is independent of $t$, then the statements in Theorem \ref{maintheorem} and Corollary \ref{corollary} remain true without assuming condition (P).
\end{theorem}

The first novelty of our paper is the fact that we present a priori bounds for very general parabolic equations with nonlinear boundary condition and involving nonlinearities that fulfill nonstandard growth conditions with a variable exponent function $p$ depending on time and space. In order to prove such bounds we obtain several results of independent interest. Indeed, although we were looking intensively in the literature, we could not find a version of the Gagliardo-Nirenberg inequality proved in Theorem \ref{multiplicative_inequalities}(2), which we needed to get the parabolic embedding stated in Proposition \ref{energy_estimate} with the critical exponent
\begin{align*}
    p^*=p\frac{N+2}{N}-\frac{2}{N}, \quad p>1.
\end{align*}
From the proof of Proposition \ref{energy_estimate} we directly deduce that $p^*$ is indeed optimal. It seems that such a critical exponent for parabolic boundary estimates is not known so far even in the constant exponent case. 

Another novelty of this work is a modified technique in order to obtain a suitable time regularization corresponding to \eqref{problem}. This leads to a new equivalent weak formulation based on so-called smoothing operators, which replace the well-known Steklov averages in the constant exponent case. Note that in our approach the log-H\"{o}lder continuity (P) is only
required for the time regularization. It is not needed for the estimates that are derived from the basic truncated energy estimates in Section \ref{Section_Energy_estimates}, here
continuity of $p$ is sufficient. In the case that $p$ does not depend on $t$ we can drop the log-H\"{o}lder continuity condition. Here one can use the well-known Steklov averaging technique, and it is sufficient to merely assume continuity of the function $p$. The present work can be seen as a nontrivial generalization of the elliptic case studied by the authors in \cite{Winkert-Zacher-2012}, \cite{Winkert-Zacher-2015} to the parabolic one.

As mentioned in the beginning, in recent years there has been a growing interest in the study of elliptic and parabolic problems involving nonlinearities that have nonstandard growth. 
Local boundedness and interior H\"older continuity of weak solutions to parabolic equations of the form
\begin{align}\label{parabolic_system}
    u_t-\divergenz \left(\left| \nabla u\right|^{p(t,x)-2}\nabla u \right)=0
\end{align}
have been proved by Xu-Chen \cite[Theorems 2.2 and 2.3]{Xu-Chen-2006}, where $p: [0,T)\times \Omega  \to \R$ is a measurable function satisfying
\begin{align}\label{log_chen}
    1 <p_1\leq p(t,x)\leq p_2<\infty, \ \left| p(t,x)-p(s,y)\right| \leq \frac{C_1}{\log\left(|x-y|+C_2|t-s|^{p_2}\right)^{-1}}
\end{align}
for any$(t,x), (s,y) \in [0,T)\times \Omega$ such that $|x-y|<\frac{1}{2}$ and $|t-s|<\frac{1}{2}$ with positive constants $p_1, p_2, C_1, C_2$. The idea in the proof is to apply a modified version of Moser's iteration. Note that the second inequality in \eqref{log_chen} is different from ours stated in (P). B\"ogelein-Duzaar \cite{Bogelein-Duzaar-2012} established local H\"older continuity of the spatial gradient of weak solutions to the parabolic system
\begin{align*}
    u_t-\divergenz \left(a(t,x)\left| \nabla u\right|^{p(t,x)-2}\nabla u \right)=0,
\end{align*}
in the sense that $\nabla u \in C^{0; \frac{\alpha}{2}, \alpha}_{\loc}$ for some $\alpha \in (0,1]$ provided the functions $p$ and $a$ satisfy a H\"older continuity property. An extension of this result to systems with nonhomogenous right-hand sides
of the form
\begin{align}\label{parabolic_system_2}
    u_t-\divergenz \left(a(t,x)\left| \nabla u\right|^{p(t,x)-2}\nabla u \right)=\divergenz \left(|F|^{p(t,x)-2}F\right),
\end{align}
could be achieved by Yao \cite{Yao-2015} (see also Yao \cite{Yao-2014}). Baroni-B\"ogelein \cite{Baroni-Bogelein-2014} have shown that the spatial gradient $\nabla u$ of the solution to \eqref{parabolic_system_2} is as integrable as the right-hand side $F$, that is
\begin{align*}
    |F|^{p(\cdot)} \in L^q_{\loc} \ \Longrightarrow \ |\nabla u|^{p(\cdot)} \in L^q_{\loc} \ \text{ for any }q>1.
\end{align*}
We also mention a similar result of B\"ogelein-Li \cite{Bogelein-Li-2014} concerning higher integrability for very weak solutions to certain degenerate parabolic systems. Partial regularity for parabolic systems like \eqref{parabolic_system} has been obtained by Duzaar-Habermann in \cite{Duzaar-Habermann-2012}.

Global and local in time $L^\infty$-bounds for weak solutions in suitable Orlicz-Sobolev spaces to the following anisotropic parabolic equations 
\begin{align*}
    \begin{cases}
	u_t-\displaystyle\sum_i D_i \left[ a_i(z,u)\left|D_iu\right|^{p_i(z)-2}D_iu+b_i(z,u)\right]+d(z,u)=0 \ \  \text{in }(0,T]\times \Omega,\\
	u=0 \ \ \text{on }\Gamma_T, \ u(0,x)=u_0(x) \text{ in }\Omega,
    \end{cases}
\end{align*}
with $z=(t,x)\in (0,T]\times \Omega$ 
 has been derived by Antontsev-Shmarev \cite{Antontsev-Shmarev-2009}. Concerning existence results to certain problems involving nonlinearity terms with $p(t,x)$-structure conditions we refer to the papers of Alkhutov-Zhikov \cite{Alkhutov-Zhikov-2010-b}, Antontsev \cite{Antontsev-2011}, Antontsev-Chipot-Shmarev \cite{Antontsev-Chipot-Shmarev-2013}, Antontsev-Shmarev \cite{Antontsev-Shmarev-2006}, \cite{Antontsev-Shmarev-2012}, \cite{Antontsev-Shmarev-2011}, \cite{Antontsev-Shmarev-2010-b}, Bauzet-Vallet-Wittbold-Zimmermann \cite{Bauzet-Vallet-Wittbold-Zimmermann-2013}, Guo-Gao \cite{Guo-Wenjie-2011}, Zhikov \cite{Zhikov-2011} and the references therein. We also mention the recent monograph of Antontsev-Shmarev \cite{Antontsev-Shmarev-2015} about several results to evolution partial differential equations with nonstandard growth conditions.

In the stationary case with $p=p(x)$ merely continuous, the authors of this manuscript established global a priori bounds for weak solutions to equations of the form
\begin{align}\label{elliptic_case}
    -\divergenz \mathcal{A} (x,u,\nabla u)  = \mathcal{B}(x,u,\nabla u)  \ \text{ in } \Om, \qquad \mathcal{A} (x,u,\nabla u)\cdot \nu  = \mathcal{C}(x,u)  \ \text{ on } \Gamma,
\end{align}
involving nonlinearities with suitable $p(x)$-structure conditions via De Giorgi iteration combined with localization, see \cite{Winkert-Zacher-2012}, \cite{Winkert-Zacher-2015}. Local boundedness of solutions to the
equation
\begin{align*}
    -\divergenz \mathcal{A} (x,u,\nabla u) & = \mathcal{B}(x,u,\nabla u)  \quad \text{ in } \Om,
\end{align*}
has been studied by Fan-Zhao \cite{Fan-Zhao-1999} and Gasi{\'n}ski-Papageorgiou (see \cite[Proposition 3.1]{Gasinski-Papageorgiou-2011}) proved global a priori bounds for weak solutions to the equation
\begin{align*}
    -\Delta_{p(x)} u  = g(x,u) \ \text{ in } \Om, \qquad  \frac{\partial u}{\partial \nu} = 0  \ \text{ on } \Gamma,
\end{align*}
where the Carath\'{e}odory function $g: \Om \times \R \to \R$ satisfies a subcritical growth condition and $p \in C^1(\overline{\Omega})$ with $1<\min_{x \in \overline{\Omega}}p(x)$. We also mention the works of You \cite{You-1998} ($C^\alpha$-regularity) and Skrypnik \cite{Skrypnik-2011} (regularity near a nonsmooth boundary) concerning parabolic equations with nonstandard growth. Existence results for $p(x)$-structure equations from different angles ($L^1$-data, blow up, anisotropic) can be found, for example in the papers of Antontsev-Shmarev \cite{Antontsev-Shmarev-2010}, Bendahmane-Wittbold-Zimmermann \cite{Bendahmane-Wittbold-Zimmermann-2010} and Pinasco \cite{Pinasco-2009}, see also the references therein.

Finally, $L^{\infty}$-estimates for solutions of \eqref{elliptic_case} in case $p(x)\equiv p$ with $q_1(x)=q_2(x)\equiv p$ have been established by the first author in \cite{Winkert-2010},\cite{Winkert-2010-b} following Moser's iteration technique (for constant $p$ see also Pucci-Servadei \cite{Pucci-Servadei-2008}).

The paper is organized as follows. Section \ref{Section_Preliminaries} collects some basic properties of the corresponding function spaces, states new interpolation inequalities and provides certain parabolic embedding results, which will be used in later considerations. In Section \ref{Section_smoothing_operators} we introduce associated smoothing operators to derive a regularized weak formulation of \eqref{problem}. Based on this, in Section \ref{Section_Energy_estimates} we prove truncated energy estimates and give the complete proof of Theorem \ref{maintheorem} by applying De Giorgi iteration along with localization.

\section{Preliminaries and hypotheses}\label{Section_Preliminaries}
Let $\Omega \subset \R^N$ be a bounded domain, $T>0$ and $Q_T=(0,T)\times \Omega$. For $p \in C(\overline{Q}_T)$ we denote by $L^{p(\cdot,\cdot)}(Q_T)$ the variable exponent Lebesgue space which is defined by
\begin{align*}
    L^{p(\cdot,\cdot)}(Q_T) = \left \{u ~ \Big | ~ u: Q_T \to \R \text{ is measurable and } \int_{Q_T} |u|^{p(t,x)}dx dt< +\infty \right \}
\end{align*}
equipped with the Luxemburg norm
\begin{align*}
    \|u\|_{L^{p(\cdot,\cdot)}(Q_T)} = \inf \left \{ \tau >0 :  \int_{Q_T} \left |\frac{u(t,x)}{\tau} \right |^{p(t,x)}dx dt \leq 1  \right \}.
\end{align*}
It is well known that $L^{p(\cdot,\cdot)}(Q_T)$ is a reflexive Banach space provided that $p^-:=\min_{\overline{Q}_T}p>1$. For more information and basic properties on variable exponent spaces we refer the reader to the papers of Fan-Zhao \cite{Fan-Zhao-2001}, Kov{\'a}{\v{c}}ik-R{\'a}kosn{\'{\i}}k \cite{Kovacik-Rakosnik-1991} and the monograph of 
Diening-Harjulehto-H{\"a}st{\"o}-R\r u\v zi\v cka\cite{Diening-Harjulehto-Hasto-Ruzicka-2011}.

The next result concerns the Gagliardo-Nirenberg multiplicative embedding inequality. First we state the following proposition on a version of a fractional Gagliardo-Nirenberg inequality (see Hajaiej-Molinet-Ozawa-Wang \cite[Proposition 4.2]{Hajaiej-Molinet-Ozawa-Wang-2011}).
\begin{proposition}\label{GN_fractional}
    Let $1<\hat{p},p_0,p_1<\infty, s,\hat{s}_1 \ge 0, 0 \leq \theta \leq 1$ and denote by $H^{s}_{\hat{p}}(\R^N):=(I-\Delta)^{-\frac{s}{2}}L^{\hat{p}}(\R^N)$ the Bessel potential space. Then there exists a positive constant $\tilde{C}$ such that the inequality
    \begin{align*}
	\|u\|_{H^{s}_{\hat{p}}(\R^N)} \leq \tilde{C}\|u\|^{\theta}_{H^{\hat{s}_1}_{p_1}(\R^N)}\|u\|^{1-\theta}_{L^{p_0}(\R^N)}
    \end{align*}
    holds if
    \begin{align*}
	\frac{N}{\hat{p}}-s=\theta \left( \frac{N}{p_1}-\hat{s}_1 \right)+(1-\theta) \frac{N}{p_0}, \quad \text{and} \quad  s \leq \theta \hat{s}_1.
    \end{align*}
\end{proposition}

\begin{remark} \label{GN_Omega}
{\em Let $\Omega\subset \R^N$ be a bounded domain with Lipschitz boundary. Then the statement of Proposition \ref{GN_fractional} remains true when replacing $\R^N$ by $\Omega$ and restricting $ s,\hat{s}_1$ to 
the interval $[0,1]$. This follows from Proposition \ref{GN_fractional} by means of extension (from $\Omega$ to
the whole space $\R^N$) and restriction. Recall that for any bounded Lipschitz domain $\Omega$ there exists a bounded 
linear extension operator from $H^1_p(\Omega)$ to $H^1_p(\R^N)$ (see e.g.\,Adams \cite{Adams-1975}) and that this property
carries over to the case of Bessel potential spaces $H^s_p$ with $s\in [0,1]$, by interpolation.}
\end{remark}

With the help of Proposition \ref{GN_fractional} and Remark \ref{GN_Omega} we can now obtain the subsequent two
interpolation (and trace) inequalities. The first one is well known, whereas we could not find any source for the second 
inequality, which is of vital importance with regard to sharp boundary estimates.  

\begin{theorem}\label{multiplicative_inequalities}
    Let $\Omega \subset \R^N$, $N> 1$, be a bounded domain with Lipschitz boundary $\Gamma:=\partial \Omega$ and let $u \in \W1p$ with $1< p<\infty$.
    \begin{enumerate}
      \item
	  For every fixed $s_1\in (1,\infty)$ there exists a constant $C_\Omega>0$ depending only upon $N, p$ and $s_1$  such that
	  \begin{align*}
	      \|u\|_{L^{q_1}(\Omega)} \leq C_\Omega \|u\|_{\W1p}^{\alpha_1} \|u\|_{L^{s_1}(\Omega)}^{1-\alpha_1},
	  \end{align*}
	  where $\alpha_1 \in [0,1]$ and $q_1\in (1,\infty)$ are linked by
	  \begin{align*}
	      & \frac{N}{q_1}=\alpha_1 \left( \frac{N}{p}-1 \right)+(1-\alpha_1) \frac{N}{s_1}.
	  \end{align*}
	  \item
	  For every fixed $s_2\in (1,\infty)$ there exists a constant $C_\Gamma>0$ depending only upon $N, p$ and $s_2$  such that
	  \begin{align*}
	      \|u\|_{L^{q_2}(\Gamma)} \leq C_\Gamma \|u\|_{\W1p}^{\alpha_2} \|u\|_{L^{s_2}(\Omega)}^{1-\alpha_2},
	  \end{align*}
	  where $\alpha_2 \in [0,1]$ and $q_2 \in (1,\infty)$ are linked by
	  \begin{align*}
	      \frac{N-1}{q_2}=\alpha_2 \left( \frac{N}{p}-1 \right)+(1-\alpha_2) \frac{N}{s_2}\qquad \text{and} \qquad \alpha_2>\frac{1}{q_2}.
	  \end{align*}
    \end{enumerate}

\end{theorem}

\begin{proof}
    We may apply Proposition \ref{GN_fractional} and Remark \ref{GN_Omega} with $s=0, \hat{p}=q_1, \hat{s}_1=1, p_1=p$, $p_0=s_1$ and $\alpha_1=\theta$. This yields the assertion of (1). Let us prove part (2).
    Since $\alpha_2>\frac{1}{q_2}$ we may fix a real number $r$ such that $\frac{1}{q_2}<r<\alpha_2$. Then we choose the number $q$ such that
    \begin{align}\label{value_q}
	\frac{r}{N}-\frac{1}{q} = - \frac{N-1}{Nq_2}.
    \end{align}
    From (\ref{value_q}) we see that $rq<N$ and
    \begin{align*}
	q=\frac{Nq_2}{rq_2+N-1}.
    \end{align*}
    Due to $\frac{1}{q_2}<r$ we have $q<q_2$ and since $N>1$ we derive $rq>1$
    thanks to the representation in (\ref{value_q}). Then, the embedding
    \begin{align}\label{Besov_boundary}
	F^{r}_{q2}(\Om) \hookrightarrow B^{r-\frac{1}{q}}_{qq}(\Gamma)
    \end{align}
    is continuous (see Triebel \cite[Section 3.3.3]{Triebel-1983}), where $B^r_{qq}$ denotes the Besov space, which coincides with the Sobolev Slobodeckij space $W^{r}_{q}$ ($r\in (0,1)$) and $F^r_{q2}$ stands for the Lizorkin-Triebel space which coincides with the Bessel potential space $H^r_q$ (see Triebel \cite[Section 2.3.5]{Triebel-1983}). In Triebel \cite[Section 3.3.3]{Triebel-1983},  a $C^\infty$-domain is required, but it is known that if $r=m+\iota$ with $m \in \N_0$ and $0 \leq \iota<1$, the embedding is still valid if $\Gamma \in C^{m,1}$. Since in our case $r<1$ we only need a Lipschitz
    boundary, that means $\Gamma \in C^{0,1}$. By virtue of the Sobolev embedding theorem for fractional order spaces it follows
    \begin{align}\label{fractional_sobolev}
        B^{r-\frac{1}{q}}_{qq}(\Gamma) \hookrightarrow L^{q_2}(\Gamma) \text{ for } q \leq q_2 \leq q^*  \text{ with }
        q^*=
        \begin{cases}
            \frac{(N-1)q}{N-rq} \quad & \text{if }rq<N,\\
            \tilde{q} \in [q,\infty) & \text{if }rq \geq N,
        \end{cases}
    \end{align}
    (see Adams \cite[Theorem 7.57]{Adams-1975}). Combining (\ref{value_q}), (\ref{Besov_boundary}) and (\ref{fractional_sobolev}) we find a positive constant $\hat{C}_1$ such that
    \begin{align}\label{GN_1}
	\|u\|_{L^{q_2}(\Gamma)} \leq \hat{C}_1 \|u\|_{F^{r}_{q2}(\Omega)} \quad \text{with} \quad  \frac{r}{N}-\frac{1}{q} = - \frac{N-1}{Nq_2}.
    \end{align}
    Now we may apply Proposition \ref{GN_fractional} and Remark \ref{GN_Omega} with $s=r, p=q, s_1=1, p_1=p$ and $p_0=s_2$ which results in
    \begin{align*}
	\|u\|_{H^{r}_{q}(\Omega)} \leq C\|u\|^{\theta}_{W^{1,p}(\Omega)}\|u\|^{1-\theta}_{L^{s_2}(\Omega)}
    \end{align*}
    with
    \begin{align}\label{GN_3}
	r-\frac{N}{q}=\theta \left(1- \frac{N}{p}\right)+(1-\theta) \left(-\frac{N}{s_2}\right) \qquad \text{and} \qquad r \leq \theta.
    \end{align}
    Since $H^r_q=F^r_{q2}$ we obtain the assertion in (2) from (\ref{GN_1})--(\ref{GN_3}) with $\alpha_2=\theta$.
\end{proof}

\newpage

\begin{remark}~
 {\em   \begin{enumerate}
	\item[(i)]  
	    If $p \neq \frac{Ns_1}{N+s_1}$ and $p \neq \frac{Ns_2}{N+s_2}$, respectively, the exponents $\alpha_1$ and $\alpha_2$ are given by
	    \begin{align*}
		& \alpha_1=\left(\frac{1}{s_1}-\frac{1}{q_1} \right) \left( \frac{1}{N}-\frac{1}{p}+\frac{1}{s_1}\right)^{-1},\\
		& \alpha_2=\left(\frac{1}{s_2}-\frac{N-1}{Nq_2} \right) \left( \frac{1}{N}-\frac{1}{p}+\frac{1}{s_2}\right)^{-1}.
	    \end{align*}
	\item[(ii)]
	    Note that in the second part of Theorem \ref{multiplicative_inequalities}, the choice $s_2=q_2=p$ is not admissible, as this leads to $\alpha_2=\frac{1}{p}$, so that the condition $\alpha_2>\frac{1}{q_2}$ is violated. However, the theorem still provides a similar estimate of the $L^p(\Gamma)$-norm from above in terms of the $W^{1,p}(\Omega)$- and $L^p(\Omega)$-norm. In fact, take $q_2=p+\varepsilon$ with small $\varepsilon>0$ and simply apply H\"older's inequality and the second part of Theorem \ref{multiplicative_inequalities} to see this. We refer to a paper of the first author \cite[Proof of Proposition 2.1]{Winkert-2014} for a similar result.
    \end{enumerate}}
\end{remark}

As a consequence of Theorem \ref{multiplicative_inequalities} we obtain two parabolic embedding inequalities  which will be useful in later considerations. The first one should be well known, see e.g.\,Chapter I in DiBenedetto \cite{DiBenedetto-1993}, which contains several variants of it (e.g.\,in the special case of vanishing boundary traces). However, we could not find any reference for the second one, which plays an important role in deriving optimal parabolic boundary estimates.    

\begin{proposition}\label{energy_estimate}
   Let $\Omega \subset \R^N$, $N> 1$, be a bounded domain with Lipschitz boundary $\Gamma:=\partial \Omega$. Let $T>0$ and $1<p<\infty$.
    \begin{enumerate}
	\item
	    There exists a constant $C_\Omega>0$ which is independent of $T$ such that
	    \begin{align*}
		& \int_0^{T} \int_\Omega |u(t,x)|^{q_1} dx dt\\
		& \leq C_\Omega^{q_1} \left ( \int_0^{T} \int_{\Omega} |\nabla u(t,x)|^p dx dt+\int_0^{T} \int_{\Omega} |u(t,x)|^p dx dt \right )\\
		& \quad \times \left (\esssup_{0<t<T} \int_\Omega |u(t,x)|^2 dx \right )^{\frac{p}{N}}
	    \end{align*}
	    for all $u \in L^{\infty}\left([0,T];L^2(\Om)\right)\cap L^p\left([0,T];W^{1,p}(\Omega)\right)$ with the exponent
	    \begin{align*}
		q_1=p \frac{N+2}{N}.
	    \end{align*}
	\item
	   There exists a constant $C_\Gamma>0$ which is independent of $T$ such that
	    \begin{align*}
		& \int_0^{T} \int_{\Gamma} |u(t,x)|^{q_2} d\sigma dt\\
		& \leq C_\Gamma^{q_2} \left ( \int_0^{T} \int_{\Omega} |\nabla u(t,x)|^p dx dt+\int_0^{T} \int_{\Omega} |u(t,x)|^p dx dt \right )\\
		& \quad \times\left (\esssup_{0<t<T} \int_\Omega |u(t,x)|^2 dx \right )^{\frac{p-1}{N}}
	    \end{align*}
	    for all $u \in L^{\infty}\left([0,T];L^2(\Om)\right)\cap L^p\left([0,T];W^{1,p}(\Omega)\right)$ with the exponent
	    \begin{align*}
		q_2=p\frac{N+2}{N}-\frac{2}{N}.
	    \end{align*}
    \end{enumerate}
\end{proposition}

\begin{proof}
    In order to prove the first part we may apply Theorem \ref{multiplicative_inequalities}(1) to the function $x \mapsto u(t,x)$ for a.a.\,$t \in (0,T)$ for $s_1=2$ and $q_1=p \frac{N+2}{N}$, which means that $\alpha_1 =\frac{p}{q_1}$.
    Taking the $q_1^{\text{th}}$-power of this inequality and integrating over $(0,T)$ yields
    \begin{align*}
	& \int_0^{T} \int_\Omega |u(t,x)|^{q_1} dx dt \\
	& \leq C_\Omega^{q_1} \int_0^{T} \left [\left (\int_{\Omega} |\nabla u(t,x)|^p dx +\int_{\Omega} |u(t,x)|^p dx dt \right ) \left (\int_\Omega |u(t,x)|^2 dx \right )^{\frac{p}{N}}\right ]dt\\
	& \leq C_\Omega^{q_1} \left ( \int_0^{T} \int_{\Omega} |\nabla u(t,x)|^p dx dt+\int_0^{T} \int_{\Omega} |u(t,x)|^p dx dt \right )\\
	& \qquad \quad \times \left (\esssup_{0<t<T} \int_\Omega |u(t,x)|^2 dx \right )^{\frac{p}{N}}.
    \end{align*}
    The second part can be proven similarly. We apply again Theorem \ref{multiplicative_inequalities}(2) to the function $x \mapsto u(t,x)$ for a.a.\,$t \in (0,T)$ for $s_2=2$ and $q_2=p \frac{N+2}{N}-\frac{2}{N}$ which gives $\alpha_2 =\frac{p}{q_2}>\frac{1}{q_2}$. Taking the $q_2^{\text{th}}$-power of this inequality and integrating over $(0,T)$ we obtain
    \begin{align*}
	& \int_0^{T} \int_\Gamma |u(t,x)|^{q_2} d\sigma dt \\
	& \leq C_\Gamma^{q_2} \int_0^{T} \left [\left (\int_{\Omega} |\nabla u(t,x)|^p dx +\int_{\Omega} |u(t,x)|^p dx dt \right ) \left (\int_\Omega |u(t,x)|^2 dx \right )^{\frac{p-1}{N}}\right ]dt\\
	& \leq C_\Gamma^{q_2} \left ( \int_0^{T} \int_{\Omega} |\nabla u(t,x)|^p dx dt+\int_0^{T} \int_{\Omega} |u(t,x)|^p dx dt \right )\\
	& \qquad \quad \times  \left (\esssup_{0<t<T} \int_\Omega |u(t,x)|^2 dx \right )^{\frac{p-1}{N}}.
    \end{align*}
\end{proof}


The following lemma concerning the geometric convergence of sequences of numbers will be needed for the De Giorgi iteration arguments below. It can be found in Ho-Sim \cite[Lemma 4.3]{Ho-Sim-2015} . The case $\delta_1=\delta_2$ is contained in Lady{\v{z}}enskaja-Solonnikov-Ural{\cprime}ceva \cite[Chapter II, Lemma 5.6]{Ladyzenskaja-Solonnikov-Uralceva-1968}, see also DiBenedetto \cite[Chapter I, Lemma 4.1]{DiBenedetto-1993}.
\begin{lemma}\label{lemma_geometric_convergence}
    Let $\{Y_n\}, n=0,1,2,\ldots,$ be a sequence of positive numbers, satisfying the recursion inequality
    \begin{align*}
        Y_{n+1} \leq K b^n \left (Y_n^{1+\delta_1}+ Y_n^{1+\delta_2} \right ) , \quad n=0,1,2, \ldots,
    \end{align*}
    for some $b>1,\,K>0$ and $\delta_2\geq \delta_1>0$. If
    \begin{align*}
        Y_0 \leq \min \left(1,(2K)^{-\frac{1}{\delta_1}} b^{-\frac{1}{\delta_1^2}}\right)
    \end{align*}
    or
    \begin{align*}
        Y_0 \leq \min  \left((2K)^{-\frac{1}{\delta_1}} b^{-\frac{1}{\delta_1^2}}, (2K)^{-\frac{1}{\delta_2}}b^{-\frac{1}{\delta_1 \delta_2}-\frac{\delta_2-\delta_1}{\delta_2^2}}\right),
    \end{align*}
    then $Y_n \leq 1$ for some $n \in \N \cup \{0\}$. Moreover,
    \begin{align*}
        Y_n \leq \min \left(1,(2K)^{-\frac{1}{\delta_1}} b^{-\frac{1}{\delta_1^2}} b^{-\frac{n}{\delta_1}}\right), \quad \text{ for all }n \geq n_0,
    \end{align*}
    where $n_0$ is the smallest $n \in \N \cup \{0\}$ satisfying $Y_n \leq 1$. In particular, $Y_n \to 0$ as $n \to \infty$.
\end{lemma}

Throughout the paper by $M_i, \tilde{M}_j$ $i,j=1,2, \ldots$ we mean positive constants depending on the given data and the Lebesgue measure on $\R^N$ is denoted by $|\cdot|_N$.

\section{Smoothing operators and regularized weak formulation}\label{Section_smoothing_operators}

Let $\rho\geq 0$ be in $C^\infty_0(\R^N)$, even, $\int_{\R^N} \rho dx=1$ and $\supp \rho =B(0,1)$. Define for $h>0$
\begin{align*}
    (S_h w)(x):= \frac{1}{h^N} \int_{\R^N} \rho \left( \frac{x-x'}{h}\right )w(x')dx', \quad x\in \R^N, w\in L^1_{\loc}(\R^N).
\end{align*}
Let $T> 0$, $e_1(t)=e^{-t}, t\ge 0$ and set
\begin{align*}
    & (\tau_h w)(t):= \frac{1}{h} \int_0^t e_1 \left( \frac{t-t'}{h} \right) w(t')dt', \quad 0 \le t \le T, w \in L^1((0,T)),\\
    & \left(\tau_h^* w\right)(t):= \frac{1}{h} \int_t^T e_1 \left( \frac{t'-t}{h} \right) w(t')dt', \quad 0 \le t \le T, w \in L^1((0,T)).
\end{align*}
Note that Fubini's theorem implies
\begin{equation*}
    \int_0^T v(t)(\tau_h^*w)(t)\,dt=\int_0^T(\tau_h v)(t)w(t)\,dt,\quad v,w\in L^1((0,T)).
\end{equation*}

Let $\Omega \subset \R^N$ be a bounded domain with Lipschitz boundary $\Gamma$ and let $p \in C(\overline{Q}_T)$ be such that $\inf_{\overline{Q}_T} p>1$ satisfying the log-H\"older condition stated in (P). By Diening-Harjulehto-H{\"a}st{\"o}-R\r u\v zi\v cka \cite[Proposition 4.1.7]{Diening-Harjulehto-Hasto-Ruzicka-2011}, $p$ can be extended to a continuous function $\tilde{p}$ on $[0,T] \times \R^N$ which fulfills $\inf_{[0,T] \times \R^N} \tilde{p}>1$ and satisfies the log-H\"older condition (P) on $[0,T] \times \R^N$. Set
\begin{align*}
    \tilde{\mathcal{V}}:=\left\{\psi \in W^{1,2}\left([0,T];L^2(\R^N)\right): |\nabla \psi| \in L^{\tilde{p}(\cdot,\cdot)}([0,T] \times \R^N) \right\}.
\end{align*}

For $h>0$, let $E_h$ be a bounded linear extension operator from $\mathcal{V}$ into $\tilde{\mathcal{V}}$
whose range is contained in the set of measurable functions that vanish almost everywhere outside of $(0,T)\times {\Omega_h}$ where $\Omega_h=\{x\in \R^N: \mbox{dist}(x,\Omega)<h^\gamma\}$ with $\gamma>2$ being fixed. Such an operator can be constructed as in Diening-Harjulehto-H{\"a}st{\"o}-R\r u\v zi\v cka \cite[Theorem 8.5.12]{Diening-Harjulehto-Hasto-Ruzicka-2011} using the log-H\"older condition of $\tilde{p}$ and by means of a suitable cut-off function. Here the construction of the operator can be made in such a way that $E_h$ also maps $L^\infty((0,T)\times \Omega)$ boundedly into $ L^\infty((0,T)\times \R^N)$ with
a corresponding norm bound that is uniform w.r.t.\,$h>0$.

\begin{lemma} \label{smoothingmapprop}
    Under the above assumptions the operators $\tau_h S_h  E_h, \tau_h^* S_h  E_h$ map from $\mathcal{V}$ into $\tilde{\mathcal{V}}$.  
\end{lemma}
The proof of the Lemma \ref{smoothingmapprop} can be done similarly as in Zhikov-Pastukhova \cite[Theorem 1.4]{Zhikov-Pastukhova-2010}.

By means of the smoothing operators introduced before we next derive a regularized weak formulation of \eqref{problem}. To this end, let $u \in \mathcal{W}$ be a {weak solution} ({subsolution, supersolution}) of \eqref{problem} in the sense of \eqref{weak_solution} and choose the test function $\varphi$ of the form
\begin{align*}
    \varphi(t,x)=(\tau_h^* S_h E_h \eta) (t,x), \quad (t,x) \in \Omega_T,
\end{align*}
where $\eta \in \mathcal{V}$ is nonnegative and $\eta|_{t=T}=0$. Observe that this test function is admissible by Lemma \ref{smoothingmapprop}
and since $\varphi|_{t=T}=0$. Note that the latter property implies that $\partial_t(\tau_h^* S_h E_h \eta)=\tau_h^* S_h \partial_t (E_h\eta)$.
In fact, for $w\in W^{1,2}((0,T))$ with $w|_{t=T}=0$ we have
\begin{align*}
    (\tau_h^* w)(t)=\frac{1}{h}\,\int_0^{T-t} e_1\left(\frac{s}{h}\right) w(s+t)\,ds, \quad t\in (0,T),
\end{align*}
and thus
\begin{align*}
    \partial_t (\tau_h^* w)(t) & =-\frac{1}{h} e_1\left(\frac{T-t}{h}\right)w(T)+\frac{1}{h}\,\int_0^{T-t} e_1\left(\frac{s}{h}\right) w_s(s+t)\,ds\\
    & = \frac{1}{h}\, \int_t^T e_1 \left( \frac{t'-t}{h} \right) w_{t'}(t')dt'=(\tau_h^* w_t)(t),\quad t\in (0,T).
\end{align*}

We obtain
\begin{align} \label{regf1}
    \begin{split}
        & - \into u_0 (\tau_h^* S_h E_h \eta) dx \Bigl |_{t=0}- \int_{0}^{T}\into u\, \tau_h^* S_h [(E_h \eta)_t] dx dt\\
        & + \int_{0}^{T}\into \mathcal{A}(t,x,u,\nabla u) \cdot \nabla \left(\tau_h^* S_h E_h \eta\right) dxdt\\
	& =\,(\le,\,\ge) \int_{0}^{T} \into \mathcal{B}(t,x,u,\nabla u) \left(\tau_h^* S_h E_h \eta\right) dxdt\\ 
	& \quad + \int_{0}^{T} \int_{\Gamma} \mathcal{C}(t,x,u) \left(\tau_h^* S_h E_h \eta\right) d \sigma dt.
    \end{split}
\end{align}
The first integral in (\ref{regf1}) takes the form
\begin{align*}
    \into u_0 (\tau_h^* S_h E_h \eta) dx \Bigl |_{t=0}=\frac{1}{h}\int_0^T\int_\Omega e_1\left(\frac{t}{h}\right)u_0(x) (S_h E_h \eta)(t,x)\,dx\,dt.
\end{align*}
The term involving the time derivative is rewritten as follows
\begin{align*}
    \begin{split}
	& - \int_{0}^{T}\into u\, \tau_h^* S_h [(E_h \eta)_t] dx dt\\
	& = - \int_{0}^{T}\int_{\R^N} E_h u\, \tau_h^* S_h [(E_h \eta)_t] dx dt+\int_0^T \int_{\Omega_h\setminus \Omega} E_h u\, \tau_h^* S_h [(E_h \eta)_t] dx dt\\
	& = \int_{0}^{T}\int_{\Omega_h} (\tau_h S_h E_h u)_t\, E_h \eta dx dt-\int_0^T\int_{\Omega_h\setminus \Omega} (\tau_h E_h u)_t\, S_h E_h \eta dx dt.
    \end{split}
\end{align*}
The remaining three terms in (\ref{regf1}) are reformulated using the duality of $\tau_h$ and $\tau_h^*$. Since the resulting relation does not contain a time derivative acting on the test function, the regularity assumptions on $\eta$ can be relaxed, in fact, by approximation, we may allow $\eta$ to be from the space $\mathcal{W}$ satisfying $\eta|_{t=T}=0$.

Next, let $0<t_1<t_2<T$ and choose $\eta$ of the form $\eta(t,x)=\psi(t,x)
\omega_{[t_1,t_2],\varepsilon}(t)$, where $\psi\in \mathcal{W}$ is nonnegative and $\omega:=\omega_{[t_1,t_2],\varepsilon}$ is defined by 
\begin{align*}
    \omega=
    \begin{cases}
	0 & \text{if } t \in [0,t_1-\varepsilon]\\
	\frac{1}{\varepsilon}(t-t_1+\varepsilon) & \text{if } t \in [t_1-\varepsilon,t_1]\\
	1 & \text{if } t \in [t_1,t_2]\\
	-\frac{1}{\varepsilon}(t-t_2-\varepsilon) & \text{if } t \in [t_2,t_2+\varepsilon]\\
	0 & \text{if } t \in [t_2+\varepsilon,T]
    \end{cases}
\end{align*}
assuming that $0<\varepsilon<\min\{t_1,T-t_2\}$. We insert such an $\eta$ in the reformulated version of (\ref{regf1}), send $\varepsilon\to 0$, divide then by $t_2-t_1$ and
finally send $t_2\to t_1$, thereby obtaining (relabeling $t_1$ by $t$) 
\begin{align}\label{regf3}
    \begin{split}
	& -\frac{1}{h}\int_\Omega e_1\left(\frac{t}{h}\right)u_0(x) (S_h E_h \psi)(t,x)\,dx+ \int_{\Omega} (\tau_h S_h E_h u)_t\, \psi dx-\mathcal{R}_h(u,\psi)(t)\\
	& + \into \big(\tau_h\mathcal{A}(\cdot,x,u,\nabla u)\big)\big|_t \cdot \nabla ( S_h E_h \psi) dx\\
	& =\,(\le,\,\ge) \into \big(\tau_h\mathcal{B}(\cdot,x,u,\nabla u)\big)\big|_t ( S_h E_h \psi) dx\\
	& \quad + \int_{\Gamma} \big(\tau_h \mathcal{C}(\cdot,x,u)\big)\big|_t (S_h E_h \psi) d \sigma,
    \end{split}
\end{align}
for a.a. $t \in (0,T)$ and for all nonnegative $\psi\in \mathcal{W}$ where
\begin{align*}
    \mathcal{R}_h(u,\psi)(t)=\int_{\Omega_h\setminus \Omega} (\tau_h E_h u)_t\, S_h E_h \psi dx-\int_{\Omega_h\setminus \Omega} (\tau_h S_h E_h u)_t\, E_h \psi dx.
\end{align*}
(\ref{regf3}) is an appropriate regularized version of the weak formulation (\ref{weak_solution}). It will be used in the following section for deriving the basic
truncated energy estimates.

If $p$ does not depend on $t$ and we merely assume that $p\in C(\close)$ the well-known Steklov averages can be used as in the constant exponent case to regularize the weak formulation in time. Indeed, defining for $v \in L^1(Q_T)$ its Steklov average by
\begin{align*}
    v_h(t,x)=\frac{1}{h} \int^{t+h}_t v(s,x)ds,
\end{align*}
we have the following result due to Alkhutov-Zhikov \cite[Lemma 5.1]{Alkhutov-Zhikov-2010} and Zhikov-Pastukhova \cite[Theorem 1.4]{Zhikov-Pastukhova-2010}.
\begin{proposition}
    Let $p$ be a function on $\Omega$ satisfying $p(x)\geq 1$ for all $x \in \Omega$. Then $v_h \to v$ in $L^{p(\cdot)}(Q_{T-\delta})$ as $h \to 0$ for any $v \in L^{p(\cdot)}(Q_{T-\delta})$ and $\delta>0$.
\end{proposition}

\section{Truncated energy estimates and proof of Theorem \ref{maintheorem}}\label{Section_Energy_estimates}

We begin this section with suitable truncated energy estimates for subsolutions and supersolutions of (\ref{problem}). First, we state the subsolution case.

\begin{proposition}\label{proposition_energy_subsolution}
    Let the assumptions in (H) and (P) be satisfied and suppose that $u_0 \in L^2(\Omega)$ is essentially bounded above in $\Omega$. Set $q_1^+=\max_{[0,T]\times \bar{\Omega}} q_1$. Then for any weak subsolution $u \in \mathcal{W}$ of (\ref{problem}) and any $\kappa$ fulfilling the condition
    \begin{align*}
	\kappa \geq \tilde{\kappa}:= \max \left \{1,\esssup_{\Omega}u_0 \right \},
    \end{align*}
    there holds
    \begin{align*}
	& \esssup_{t \in (0,T_0)} \int_{A_{\kappa}(t)} (u-\kappa)^2 dx+ \int_{0}^{T_0} \int_{A_{\kappa}(t)} |\nabla u|^{p(t,x)} dx dt\\
	& \leq M_1 \int_{0}^{T_0} \int_{A_{\kappa}(t)} u^{q_1(t,x)} dx dt+ M_2 \int_{0}^{T_0} \int_{\Gamma_{\kappa}(t)} u^{q_2(t,x)} d \sigma dt
    \end{align*}
    for every $T_0 \in (0,T]$ with
    \begin{align*}
	A_{\kappa}(t)=\{x\in \Om : u(t,x)>\kappa\}, \qquad \Gamma_{\kappa}(t)=\{x \in \Gamma: u(t,x)>{\kappa}\}, \qquad t \in (0,T_0],
    \end{align*}
    and with positive constants $M_1=M_1(q_1^+,a_3,a_4,a_5,b_0,b_1,b_2)$ as well as $M_2=M_2(a_3,c_0,c_1)$.
\end{proposition}

\begin{proof}
  {\bf (I) Regularized testing.}    
    Let $u \in \mathcal{W}$ be a weak subsolution of (\ref{problem}) and fix ${\kappa} \geq \tilde{\kappa}$. For $h>0$ we set $\Phi_h(u)= \tau_h S_h E_h u$. Letting $\lambda>0$ we further define the truncations $T_\lambda(y)=\min(y,\lambda)$ and $[y]_\kappa^+:=\max(y-\kappa,0)$, $y\in \R$.  We take in (\ref{regf3})
    the test function $\psi=T_\lambda([\Phi_h(u)]_\kappa^+)$, which belongs to the space $\mathcal{W}$, see Le \cite[Lemma 3.2]{Le-2009}. Integrating over $(0,t_0)$ where $t_0 \in (0,T_0]$ is arbitrarily fixed, we obtain
    \begin{align}\label{TR1}
	\begin{split}
	    & -\int_0^{t_0}\int_\Omega \frac{e_1(t/h)}{h}u_0(x) (S_h E_h T_\lambda([\Phi_h(u)]_\kappa^+))(t,x)\,dx dt\\
	    & \quad + \frac{1}{2}\int_{\Omega} \big(T_\lambda([\Phi_h(u)]_\kappa^+)(t_0,x)\big)^2 dx-\int_0^{t_0}\mathcal{R}_h\big(u,T_\lambda([\Phi_h(u)]_\kappa^+)\big)(t) dt\\
	    & \quad + \int_0^{t_0}\into \big(\tau_h\mathcal{A}(t,x,u,\nabla u)\big) \cdot \nabla ( S_h E_h 
	    T_\lambda( [\Phi_h(u)]_\kappa^+)) dx dt\\
	    & =\,(\le,\,\ge) \int_0^{t_0}\into \big(\tau_h\mathcal{B}(t,x,u,\nabla u)\big) ( S_h E_h 
	    T_\lambda([\Phi_h(u)]_\kappa^+)) dx dt \\
	    &  \quad+ \int_0^{t_0}\int_{\Gamma} \big(\tau_h \mathcal{C}(t,x,u)\big) (S_h E_h 
	    T_\lambda([\Phi_h(u)]_\kappa^+)) d \sigma dt.
	\end{split}
    \end{align}
    We next send $h\to 0$ in (\ref{TR1}) and make use of the approximation properties of the smoothing operators involved.

    Note first that for any $w\in C([0,t_0])$
    \begin{align*}
	\int_0^{t_0}\frac{e_1(t/h)}{h}\, w(t)\,dt \to w(0)\quad \text{as } h\to 0,
    \end{align*}
    and thus it is not difficult to see that the first term in (\ref{TR1}) tends to
    \begin{align*}
	-\int_\Omega u_0(x)T_\lambda((u(t,x)-\kappa)_+)\,dx\big|_{t=0}  =-\int_\Omega u_0(x)T_\lambda((u_0(x)-\kappa)_+)\,dx=0,
    \end{align*}
    due to $\kappa\ge \tilde{\kappa}$. Further, as $h\to 0$ we have
    \begin{align*}
	& \int_{\Omega} \big(T_\lambda([\Phi_h(u)]_\kappa^+)(t_0,x)\big)^2 dx\\  
	& \qquad \qquad \to \int_{\Omega} \big(T_\lambda((u-\kappa)_+(t_0,x))\big)^2 dx,\\
	& \int_0^{t_0}\into \big(\tau_h\mathcal{A}(t,x,u,\nabla u)\big) \cdot \nabla ( S_h E_h T_\lambda([\Phi_h(u)]_\kappa^+))\, dx dt\\  
	& \qquad \qquad \to \int_0^{t_0}\into \mathcal{A}(t,x,u,\nabla u) \cdot \nabla T_\lambda((u-\kappa)_+)\, dx dt,\\
	& \int_0^{t_0}\into \big(\tau_h\mathcal{B}(t,x,u,\nabla u)\big) ( S_h E_h T_\lambda([\Phi_h(u)]_\kappa^+))\, dx dt\\
	& \qquad \qquad \to \int_0^{t_0}\into \mathcal{B}(t,x,u,\nabla u) T_\lambda((u-\kappa)_+)\, dx dt,\\
	& \int_0^{t_0}\int_{\Gamma} \big(\tau_h \mathcal{C}(t,x,u)\big) (S_h E_h T_\lambda([\Phi_h(u)]_\kappa^+))\, d \sigma dt\\
	& \qquad \qquad \to \int_0^{t_0}\int_{\Gamma} \mathcal{C}(t,x,u) T_\lambda((u-\kappa)_+)\, d \sigma dt.
    \end{align*}
    Finally, we claim that
    \begin{equation} \label{TR2}
	\int_0^{t_0}\mathcal{R}_h\big(u,T_\lambda([\Phi_h(u)]_\kappa^+)\big)(t) dt \to 0  \quad \text{as }h\to 0.
    \end{equation}
    To see this, note first that the boundedness of $\psi=T_\lambda([\Phi_h(u)]_\kappa^+)\big)$ and the mapping properties of $E_h$ and $S_h$ imply that
    $E_h \psi$ as well as $S_h E_h \psi$ are bounded 
    uniformly w.r.t.\,$h>0$. Note also that for any $w\in L^1((0,T))$ we have
    \begin{align*}
	\partial_t(\tau_h w)(t)=\frac{1}{h}\,\big(w(t)-(\tau_h w)(t)\big),\quad \mbox{a.a.}\,t\in (0,T).
    \end{align*}
    Thus we get an estimate of the form
    \begin{align*}
	|\mathcal{R}_h(u,\psi)(t)|\le \frac{C}{h}\,\int_{\Omega_h \setminus \Omega}F_h(t,x)\,dx,\quad \mbox{a.a.}\,t\in (0,T),
    \end{align*}
    where
    \begin{align*}
	F_h=|E_h u|+|\tau_h E_h u|+|S_h E_h u|+|\tau_h S_h E_h u|
    \end{align*}
    and the constant $C$ is independent of $h$. By H\"older's inequality, it follows that
    \begin{align} \label{RHest}
	\int_0^{t_0}|\mathcal{R}_h(u,\psi)(t)|\,dt\le \frac{C}{h}\, |\Omega_h \setminus \Omega|^{1/2}\int_0^{t_0} |F_h(t,\cdot)|_{L^2(\R^N)}\,dt
    \end{align}
    Recalling the definition of $\Omega_h$ we have that $|\Omega_h \setminus \Omega|\le \tilde{C} h^{\gamma}$, where
    $\gamma>2$. Since the integral term on the right hand side of (\ref{RHest}) stays bounded for $h\to 0$, it follows that
    $\int_0^{t_0}\mathcal{R}_h(u,\psi)(t)\,dt$ tends to $0$ as $h\to 0$ as claimed in \eqref{TR2}.

    Combining the previous statements and sending the truncation parameter $\lambda\to \infty$ we conclude that for all $t_0\in (0,T_0]$
    \begin{align}\label{PES_11}
	\begin{split}
	    & \frac{1}{2}\int_{\Omega} \big((u-\kappa)_+(t_0,x)\big)^2 dx+\int_0^{t_0}\into \mathcal{A}(t,x,u,\nabla u) \cdot \nabla (u-\kappa)_+\, dx dt\\
	    & \le \int_0^{t_0}\into \mathcal{B}(t,x,u,\nabla u) (u-\kappa)_+\, dx dt +\int_0^{t_0}\int_{\Gamma} \mathcal{C}(t,x,u)(u-\kappa)_+\, d \sigma dt.
	\end{split}
    \end{align}

  {\bf (II) Employing the structure.}
    Now we may apply the structure conditions stated in (H) to the various terms in \eqref{PES_11}. Using (H1) the
    second term on the left-hand side of \eqref{PES_11} can be estimated as
    \begin{align}\label{PES_3}
	\begin{split}
	    & \int_0^{t_0}\into \mathcal{A}(t,x,u,\nabla u) \cdot \nabla (u-\kappa)_+\, dx dt\\
	    & = \int_{0}^{t_0} \int_{A_{\kappa}(t) } \mathcal{A}(t,x,u,\nabla u) \cdot \nabla u dxdt   \\
	    & \geq  \int_{0}^{t_0}\int_{A_\kappa(t)} \left (a_3 |\nabla u|^{p(t,x)}-a_4|u|^{q_1(t,x)}-a_5 \right ) dx dt \\
	    & \geq a_3  \int_{0}^{t_0} \int_{A_\kappa(t)} |\nabla u|^{p(t,x)} dx dt - (a_4+a_5)   \int_{0}^{t_0} \int_{A_\kappa(t)} |u|^{q_1(t,x)} dxdt,
	\end{split}
    \end{align}
    since $u^{q_1(t,x)}>u>1$ in $A_\kappa(t)$. 
    
    Let us next estimate the first term on the right-hand side of \eqref{PES_11} by applying the structure condition (H3) and Young's inequality with $\eps \in
    (0,1]$. This gives
    \begin{align}\label{PES_4}
        \begin{split}
	    & \int_0^{t_0}\into \mathcal{B}(t,x,u,\nabla u) (u-\kappa)_+\, dx dt\\
            & \leq  \int_{0}^{t_0}\int_{A_{{\kappa}}(t)} \left [b_0 |\nabla u|^{p(t,x)\frac{q_1(t,x)-1}{q_1(t,x)}}+b_1 |u|^{q_1(t,x)-1}+b_2 \right ]  (u-{\kappa}) dxdt\\
            & \leq b_0  \int_{0}^{t_0} \int_{A_{{\kappa}}(t)} \left [ \eps^{\frac{q_1(t,x)-1}{q_1(t,x)}} |\nabla u|^{p(t,x)\frac{q_1(t,x)-1}{q_1(t,x)}} \eps^{-\frac{q_1(t,x)-1}{q_1(t,x)}}u \right ] dxdt\\
            & \quad +(b_1+b_2)  \int_{0}^{t_0} \int_{A_{\kappa}(t)}|u|^{q_1(t,x)}dxdt\\
            & \leq b_0  \int_{0}^{t_0}\int_{A_{\kappa}(t)} \eps |\nabla u|^{p(t,x)} dxdt + b_0 \int_{0}^{t_0} \int_{A_{\kappa}(t)} \eps^{-(q_1(t,x)-1)} u^{q_1(t,x)}dxdt \\
            & \quad +(b_1+b_2)  \int_{0}^{t_0} \int_{A_{\kappa}(t)}|u|^{q_1(t,x)}dxdt \\
            & \leq \eps b_0 \int_{0}^{t_0} \int_{A_{\kappa}(t)}|\nabla u|^{p(t,x)} dxdt\\
            & \quad + \left ( b_0 \eps^{-(q_1^+-1)}+b_1+b_2\right ) \int_{0}^{t_0} \int_{A_{\kappa}(t)} u^{q_1(t,x)} dxdt .
        \end{split}
    \end{align}
    Finally, we use assumption (H4) to estimate the boundary term through
    \begin{align}\label{PES_5}
        \begin{split}
	    & \int_0^{t_0}\int_{\Gamma} \mathcal{C}(t,x,u)(u-\kappa)_+\, d \sigma dt\\
            & \leq  \int_{0}^{t_0} \int_{\Gamma_{\kappa}(t)} (c_0 |u|^{q_2(t,x)-1}+c_1)(u-{\kappa}) d \sigma dt\\
            & \leq (c_0+c_1)  \int_{0}^{t_0} \int_{\Gamma_{\kappa}(t)} u^{q_2(t,x)} d \sigma dt.
        \end{split}
    \end{align}
    Combining (\ref{PES_11})--(\ref{PES_5}) results in
    \begin{align}\label{PES_6}
    \begin{split}
        & \frac{1}{2} \into (u(t_0,x)-{\kappa})_+^2 dx +\frac{a_3}{2} \int_{0}^{t_0} \int_{A_{\kappa}(t)} |\nabla u|^{p(t,x)} dx dt\\
        &\leq \tilde{M}_1  \int_{0}^{t_0} \int_{A_{\kappa}(t)} u^{q_1(t,x)}dxdt + \tilde{M}_2  \int_{0}^{t_0} \int_{\Gamma_{\kappa}(t)} u^{q_2(t,x)}d \sigma dt,
    \end{split}
    \end{align}
    for every $t_0 \in (0,T_0]$, whereby $\eps$ was chosen such that $\eps = \min\left(1,\frac{a_3}{2b_0}\right)$ and $\tilde{M}_1=\tilde{M}_1\left(q_1^+,a_3,a_4,a_5,b_0,b_1,b_2\right)$  as well as $\tilde{M}_2=\tilde{M}_2(c_0,c_1)$.

   Since (\ref{PES_6}) holds for all $t_0\in (0,T_0]$ and the second term on the left-hand side of (\ref{PES_6}) is nonnegative,
   the assertion of the proposition follows.   
\end{proof}

Similar to Proposition \ref{proposition_energy_subsolution} we may formulate a corresponding result for supersolutions of (\ref{problem}).

\begin{proposition}\label{proposition_energy_supersolution}
    Let the assumptions in (H) and (P) be satisfied and suppose that $u_0 \in L^2(\Omega)$ is essentially bounded below in $\Omega$. Then for any weak supersolution $u \in \mathcal{W}$ of (\ref{problem}) and any $\kappa$ fulfilling the condition
    \begin{align*}
	\kappa \geq \hat{\kappa}:= \max \left \{1,-\essinf_{\Omega}u_0 \right \},
    \end{align*}
    there holds
    \begin{align*}
	& \esssup_{t \in (0,T_0)} \int_{\tilde{A}_{\kappa}(t)} (u+\kappa)^2 dx+ \int_{0}^{T_0} \int_{\tilde{A}_{\kappa}(t)} |\nabla u|^{p(t,x)} dx dt\\
	& \leq M_1 \int_{0}^{T_0} \int_{\tilde{A}_{\kappa}(t)} (-u)^{q_1(t,x)} dx dt+ M_2 \int_{0}^{T_0} \int_{\tilde{\Gamma}_{\kappa}(t)} (-u)^{q_2(t,x)} d \sigma dt
    \end{align*}
    for every $T_0 \in (0,T]$ with
    \begin{align*}
	\tilde{A}_{\kappa}(t)=\{x\in \Om : -u(t,x)>\kappa\}, \qquad \tilde{\Gamma}_{\kappa}(t)=\{x \in \Gamma: -u(t,x)>{\kappa}\}, \quad t \in (0,T_0],
    \end{align*}
    and with the same constants $M_1$ and $M_2$ as in Proposition \ref{proposition_energy_subsolution}.
\end{proposition}

\begin{proof}
    The proof is analogous to the subsolution case. Replacing $u$ by $-u$ and $u_0$ by $-u_0$, the same line
    of arguments yields the asserted estimate.
\end{proof}

Now we are in the position to prove Theorem \ref{maintheorem}.

\begin{proof}[Proof of Theorem \ref{maintheorem}]
  Our proof is divided into several parts.


    {\bf (I) Partition of unity.}
    Since $\bar{\Omega}$ is compact, for any $R>0$ there exists an open cover $\{B_j(R)\}_{j=1, \ldots, m}$ of balls
    $B_j:=B_j(R)$ with radius $R>0$ such that $\bar{\Omega}\subset \bigcup_{j=1}^m B_j(R)$. We further decompose
    the time interval as 
    \begin{align*}
        & [0,T]=\bigcup_{i=1}^l J_i \quad \text{with } J_i:=J_i(\delta)=[\delta(i-1),\delta i],
    \end{align*}
    where $l \delta=T$.
    
    Recall that
    \begin{align*}
        & p(t,x) \leq q_1(t,x)<p^*(t,x), \quad (t,x)\in [0,T] \times \overline{\Om}=\overline{Q}_T,\\
        & p(t,x)\leq q_2(t,x)<p_*(t,x), \quad (t,x) \in [0,T]\times \Gamma=\overline{\Gamma}_T.
    \end{align*}
    Clearly, since $p, q_1 \in C(\overline{Q}_T)$ and $q_2 \in
    C(\overline{\Gamma}_T)$ these functions are uniformly
    continuous on $\overline{Q}_T$ and $\overline{\Gamma}_T$. Hence,
    we may take $R>0$ and $\delta>0$ small enough such that
    \begin{align*}
        p^+_{i,j} \leq q_{1,i,j}^+ <(p^-_{i,j})^*, \qquad p^+_{i,j} \leq q_{2,i,j}^+<(p^-_{i,j})_{*},
    \end{align*}
    for $i=1, \ldots, l$ and $j=1, \ldots, m$ whereby
    \begin{align*}
        \begin{split}
            & p^+_{i,j}=\max_{(t,x) \in J_i \times (\overline{B_j} \cap \overline{\Omega})} p(t,x),
            \quad q_{1,i,j}^+=\max_{(t,x) \in J_i \times (\overline{B_j}\cap \overline{\Om})}q_1(t,x),\\
            & p^-_{i,j}=\min_{(t,x) \in J_i \times (\overline{B_j}\cap \overline{\Om})} p(t,x),
            \quad q_{2,i,j}^+=\max_{(t,x) \in J_i \times (\overline{B_j} \cap \Gamma)} q_2(t,x).
        \end{split}
    \end{align*}
    Recall that, for $s\in[1,\infty)$,
    \begin{align*}
	s^*= s\frac{N+2}{N}, \qquad s_*= s\frac{N+2}{N}-\frac{2}{N}.
    \end{align*}    
    Now we choose a partition of unity $\{\xi_j\}_{j=1}^m \subset C^{\infty}_0 (\RN)$ with respect to the open cover
    $\{B_j(R)\}_{i=j,\ldots, m}$ (see e.g.\,Rudin \cite[Theorem 6.20]{Rudin-1973}) which means
    \begin{align*}
        \supp \xi_j \subset B_j, \quad 0 \leq \xi_j \leq 1, \quad j=1,\ldots,m, \quad \text{and} \quad \sum_{j=1}^m \xi_j=1 \text{ on } \overline{\Omega}.
    \end{align*}
    Moreover, we denote by $L$ a positive constant satisfying
    \begin{align}\label{I_1}
        |\nabla \xi_j| \leq L,\quad j=1,\ldots,m.
    \end{align}
    Without loss of generality we may assume that $L>1$.


    {\bf (II) Iteration variables and basic estimates.}
    First, we set
    \begin{align*}
        \kappa_n=\kappa \left (2 -  \frac{1}{2^{n}} \right ), \quad n=0,1,2, \ldots ,
    \end{align*}
    with $\kappa \geq \max \left \{1,\esssup_{\Omega}u_0 \right \}$ specified later and put
    \begin{align*}
        Z_n:=\int_0^\delta \int_{A_{\kappa_n}(t)} (u-\kappa_n)^{q_1(t,x)}dxdt,
        \qquad \tilde{Z}_n:= \int_0^\delta \int_{\Gamma_{\kappa_n}(t)} (u-\kappa_n)^{q_2(t,x)} d \sigma dt.
    \end{align*}
    Thanks to
    \begin{align*}
        \begin{split}
            Z_n
            & \geq \int_0^\delta \int_{A_{\kappa_{n+1}}(t)} (u-\kappa_n)^{q_1(t,x)} dx dt\\
            & \geq \int_0^\delta \int_{A_{\kappa_{n+1}}(t)} u^{q_1(t,x)} \left (1-\frac{\kappa_n}{\kappa_{n+1}} \right )^{q_1(t,x)}dx dt\\
            & \geq \int_0^\delta \int_{A_{\kappa_{n+1}}(t)} \frac{1}{2^{q_1(t,x) (n+2)}} u^{q_1(t,x)}dx,
        \end{split}
    \end{align*}
    we have
    \begin{align}\label{II_1}
        \int_0^\delta \int_{A_{\kappa_{n+1}}(t)} u^{q_1(t,x)} dx dt  \leq 2^{q_{1}^+(n+2)} Z_n.
    \end{align}
    Analogously, one proves
    \begin{align}\label{II_2}
        \int_0^\delta \int_{\Gamma_{\kappa_{n+1}}(t)} u^{q_2(t,x)} d \sigma dt \leq 2^{q_2^+(n+2)}
        \tilde{Z}_n.
    \end{align}
    Due to Proposition \ref{proposition_energy_subsolution} (replacing $\kappa$ by $\kappa_{n+1} \geq \max \left \{1,\esssup_{\Omega}u_0 \right \}$ and $T_0$ by $\delta$)
    along with (\ref{II_1}) and (\ref{II_2}) we obtain
    \begin{align}\label{II_3}
        \begin{split}
            & \esssup_{t \in (0,\delta)} \int_{A_{\kappa_{n+1}}(t)} (u-\kappa_{n+1})^2 dx+ \int_{0}^{\delta} \int_{A_{\kappa_{n+1}}(t)} |\nabla (u-\kappa_{n+1})|^{p(t,x)} dx dt\\
            & \leq M_3 M_4^n (Z_n + \tilde{Z}_n),
        \end{split}
    \end{align}
    where $M_3=\max\left (M_1 2^{2q_1^+},M_2 2^{2q_2^+}\right )$ and $M_4=\max \left (2^{q_1^+}, 2^{q_2^+} \right)$.
    Additionally, it holds
    \begin{align}\label{II_4}
        \begin{split}
            \int_0^\delta |A_{\kappa_{n+1}}(t)|dt
            & \leq \int_0^\delta \int_{A_{\kappa_{n+1}}(t)} \left (\frac{u-\kappa_n}{\kappa_{n+1}-\kappa_n}\right )^{q_1(t,x)} dx dt\\
            & \leq \int_0^\delta \int_{A_{\kappa_{n}}(t)} \frac{2^{q_1(t,x)(n+1)}}{\kappa^{q_1(t,x)}} (u-\kappa_n)^{q_1(t,x)}dx dt\\
            & \leq \frac{2^{q_1^+(n+1)}}{\kappa^{q_1^-}} \int_0^\delta \int_{A_{\kappa_{n}}(t)} (u-\kappa_n)^{q_1(t,x)}dx dt \\
            & = \frac{2^{q_1^+(n+1)}}{\kappa^{q_1^-}} Z_n.
        \end{split}
    \end{align}
    Furthermore, we set
    \begin{align}\label{II_5}
	Y_n:=Z_n+\tilde{Z}_n.
    \end{align}


    {\bf (III) Estimating the gradient term in (\ref{II_3}) from below.}
    With the aid of the partition of unity from step (I) it follows
    \begin{align}\label{III_1}
        \begin{split}
            & \int_0^\delta \int_{A_{\kappa_{n+1}}(t)} |\nabla (u-\kappa_{n+1})|^{p(t,x)}dxdt \\
            & =  \int_0^\delta \int_{A_{\kappa_{n+1}}(t)} |\nabla (u-\kappa_{n+1})|^{p(t,x)} \sum_{j=1}^m \xi_j dxdt\\
            & \geq  \sum_{j=1}^m  \int_0^\delta\int_{A_{\kappa_{n+1}}(t)} \left( |\nabla (u-\kappa_{n+1})|^{p^-_{1,j}}-1\right)  \xi_j dxdt\\
            & \geq \left(\sum_{j=1}^m \int_0^\delta \int_{A_{\kappa_{n+1}}(t)} |\nabla (u-\kappa_{n+1})|^{p^-_{1,j}} \xi_j^{p^-_{1,j}} dxdt\right)\\
            & \qquad -\left(m \int_0^\delta |A_{\kappa_{n+1}}(t)|dt\right),
        \end{split}
    \end{align}
    since $\xi_j  \geq \xi_j^{p^-_{1,j}}$. In particular, from (\ref{III_1}) we conclude
    \begin{align}\label{III_2}
        \begin{split}
            & \int_0^\delta \int_{A_{\kappa_{n+1}}(t)} |\nabla (u-\kappa_{n+1})|^{p(t,x)}dx dt\\
            & \geq  \int_0^\delta \int_{A_{\kappa_{n+1}}(t)} |\nabla (u-\kappa_{n+1})|^{p^-_{1,j}} \xi_j^{p^-_{1,j}} dxdt - m \int_0^\delta |A_{\kappa_{n+1}}(t)|dt,
        \end{split}
    \end{align}
    for all $j=1, \ldots, m$.
    Combining (\ref{III_2}) and (\ref{II_3}) and using (\ref{II_4}) yields
    \begin{align}\label{III_3}
        \begin{split}
             & \esssup_{t \in (0,\delta)} \int_{A_{\kappa_{n+1}}(t)} (u-\kappa_{n+1})^2 dx
             + \int_0^\delta\int_{A_{\kappa_{n+1}}(t)} |\nabla (u-\kappa_{n+1})|^{p^-_{1,j}} \xi_j^{p^-_{1,j}} dxdt\\
             & \leq M_5 M_4^n (Z_n + \tilde{Z}_n)
        \end{split}
    \end{align}
    for any $j=1,\ldots, m$ with the positive constant $M_5=M_3+m2^{q_1^+}$. Recall that $M_4=\max \left (2^{q_1^+}, 2^{q_2^+} \right)$ (see step (II)).


    {\bf (IV) Estimating the term $Z_{n+1}$.}
    Let us now estimate $Z_{n+1}$ from above using the partition of
    unity. First, we have
    \begin{align}\label{IV_1}
        \begin{split}
            Z_{n+1} &=\int_0^\delta \int_{A_{\kappa_{n+1}}} (u-\kappa_{n+1})^{q_1(t,x)}dxdt\\
            & =\int_0^\delta\int_{A_{\kappa_{n+1}}(t)} (u-\kappa_{n+1})^{q_1(t,x)} \left ( \sum_{j=1}^m \xi_j\right )^{q_1^+}dxdt \\
            & \leq m^{q_1^+} \sum_{j=1}^m \int_0^\delta \int_{A_{\kappa_{n+1}}(t)} (u-\kappa_{n+1})^{q_1(t,x)} \xi_j^{q_{1,1,j}^+} dxdt \\
            & \leq m^{q_1^+} \sum_{j=1}^m \left [\int_0^\delta \int_{A_{\kappa_{n+1}}(t)} (u-\kappa_{n+1})^{q_{1,1,j}^+}\xi_j^{q_{1,1,j}^+}dxdt \right.\\
            & \qquad \qquad \qquad  \left. + \int_0^\delta \int_{A_{\kappa_{n+1}}(t)} (u-\kappa_{n+1})^{q_{1,1,j}^-}\xi_j^{q_{1,1,j}^-} dx dt \right ],
        \end{split}
    \end{align}
    where $q_{1,1,j}^-=\min_{(t,x) \in J_1 \times (\overline{B_j}\cap \overline{\Om})}q_1(t,x)$.
    Note that $p^-_{1,j} \leq q^-_{1,1,j} \leq q^+_{1,1,j} <(p^-_{1,j})^*$ for all $j=1,\ldots, m$.

    Now, we fix $j\in \{1,\ldots, m\}$ and assume that $r \in \{q^-_{1,1,j},q^+_{1,1,j}\}$.
    Then $p^-_{1,j} \leq r <(p^-_{1,j})^*$ and $r \leq q^+$, where $q^+=\max(q^+_1,q^+_2)$.

    By combining H\"{o}lder's inequality with Proposition \ref{energy_estimate}(1) we obtain
    \begin{align}\label{IV_2}
        \begin{split}
            & \int_0^\delta \into (u-\kappa_{n+1})_+^{r}\xi_j^{r}dxdt\\
            & \leq \int_0^\delta \left [\left (\into (u-\kappa_{n+1})_+^{(p^-_{1,j})^*} \xi_j^{(p^-_{1,j})^*}dx \right)^{\frac{r}{(p^-_{1,j})^*}}
            |A_{\kappa_{n+1}}(t)|^{1-\frac{r}{(p^-_{1,j})^*}}\right ] dt  \\
            & \leq \left [\int_0^\delta \into (u-\kappa_{n+1})_+^{(p^-_{1,j})^*} \xi_j^{(p^-_{1,j})^*}dx dt\right ]^{\frac{r}{(p^-_{1,j})^*}}
            \left [\int_0^\delta |A_{\kappa_{n+1}}(t)|dt\right ]^{1-\frac{r}{(p^-_{1,j})^*}} \\
            & \leq  \tilde{C}^{q^+} \left ( \int_0^{\delta} \int_{\Omega} \left|\nabla[(u-\kappa_{n+1})_+\xi_j]\right|^{p^-_{1,j}} dx dt\right.\\
            & \qquad \qquad  \left.+\int_0^{\delta} \int_{\Omega} (u-\kappa_{n+1})_+^{p^-_{1,j}}\xi_j^{p^-_{1,j}} dx dt \right )^{\frac{r}{(p^-_{1,j})^*}}\\
	        & \qquad \qquad \times \left (\esssup_{0<t<\delta} \int_\Omega (u-\kappa_{n+1})_+^2 dx \right )^{\frac{r}{N+2}}\left [\int_0^\delta |A_{\kappa_{n+1}}(t)|dt\right ]^{1-\frac{r}{(p^-_{1,j})^*}},
	\end{split}
    \end{align}
    where $\tilde{C}=\max(1,C_\Omega(p^-_{1,1},N), \ldots, C_\Omega(p^-_{1,m},N))$ with $C_\Omega(p^-_{1,j},N)$ being the
    constant of the energy estimate given in Proposition
    \ref{energy_estimate}(1), $j=1,\ldots,m$. Thus $\tilde{C}$ is independent of $j$. Furthermore, the right-hand side of \eqref{IV_2} can be estimated to obtain
    \begin{align}\label{IV_2b}
        \begin{split}
            & \int_0^\delta \into (u-\kappa_{n+1})_+^{r}\xi_j^{r}dxdt\\
            & \leq  M_6 \left ( \int_0^{\delta} \int_{A_{\kappa_{n+1}}(t)} |\nabla(u-\kappa_{n+1})|^{p^-_{1,j}}\xi_j^{p^-_{1,j}} dx dt\right. \\
            & \qquad \qquad \left. +\int_0^{\delta} \int_{A_{\kappa_{n+1}}(t)} u^{q_{1}(t,x)} dx dt \right.\\
            & \qquad \qquad \left.+\esssup_{0<t<\delta} \int_\Omega (u-\kappa_{n+1})_+^2 dx \vphantom{\int_0^\delta} \right )^{r \left(\frac{1}{p^-_{1,j}}\frac{N}{N+2}+\frac{1}{N+2}\right) }\\
            & \qquad \qquad \times \left [\int_0^\delta |A_{\kappa_{n+1}}(t)|dt\right ]^{1-\frac{r}{(p^-_{1,j})^*}}
	\end{split}
    \end{align}
    with $M_6=M_6(p^+, q^+,\tilde{C},L)$. Applying (\ref{III_3}), (\ref{II_1}), (\ref{I_1}), (\ref{II_4}) and (\ref{II_5}) to the right-hand side of (\ref{IV_2b}) yields
    \begin{align}\label{IV_3}
        \begin{split}
            & \int_0^\delta \into (u-\kappa_{n+1})_+^{r}\xi_j^{r}dxdt\\
            & \leq  M_6 \left ( M_5M_4^n(Z_n+\tilde{Z}_n)+2^{q_1^+(n+2)}Z_n \right )^{r \left(\frac{1}{p^-_{1,j}}\frac{N}{N+2}+\frac{1}{N+2}\right) }\\
            & \qquad \times \left [\frac{2^{q_1^+(n+1)}}{\kappa^{q_1^-}} Z_n\right ]^{1-\frac{r}{(p^-_{1,j})^*}}\\
            & \leq M_6 2^{q^+} \left ( M_5^{q^+ }\left(M_4^{q^+ } \right)^n \left( Y_n + Y_n^{q^+} \right) + \left( 2^{q_1^+q^+ } \right)^{n+2} \left( Y_n+Y_n^{q^+ }\right)\right)\\
            & \qquad \times \left [\frac{2^{q_1^+(n+1)}}{\kappa^{q_1^-}} Z_n\right ]^{1-\frac{r}{(p^-_{1,j})^*}}\\
            & \leq M_7 M_8^n \left( Y_n+Y_n^{q^+}\right) \left [\frac{2^{q_1^+(n+1)}}{\kappa^{q_1^-}} Z_n\right ]^{1-\frac{r}{(p^-_{1,j})^*}},
	\end{split}
    \end{align}
    where we have used the estimate
    \begin{align*}
	r \left(\frac{1}{p^-_{1,j}}\frac{N}{N+2}+\frac{1}{N+2}\right) \leq q^+\frac{N+1}{N+2}\leq q^+.
    \end{align*}
    Set $\eta=\max\left(\frac{q^+_{1,1,1}}{(p^-_{1,1})^*}, \ldots, \frac{q^+_{1,1,m}}{(p^-_{1,m})^*} \right)$. Then, we can estimate the last term on the right-hand side of (\ref{IV_3}) as follows
    \begin{align} \label{IV_4}
        \begin{split}
            & \left [\frac{2^{q_1^+(n+1)}}{\kappa^{q_1^-}} Z_n\right ]^{1-\frac{r}{(p^-_{1,j})^*}} \leq 2^{q_1^+(n+1)} \left (\frac{1}{\kappa^{q_1^-}} \right)^{1-\eta}  (Y_n+Y_n^{1-\eta})
        \end{split}
    \end{align}
    for $r \in \{q^+_{1,1,j},q^-_{1,1,j}\}$.

    Now we may apply (\ref{IV_3}) and (\ref{IV_4}) with $r=q_{1,1,j}^+$ and $r=q_{1,1,j}^-$, respectively, to (\ref{IV_1}) which results in
    \begin{align}\label{IV_5}
        \begin{split}
            Z_{n+1}
            & \leq m^{q_1^+} \sum_{j=1}^m \left [\int_0^\delta \int_{A_{\kappa_{n+1}}(t)} (u-\kappa_{n+1})^{q_{1,1,j}^+}\xi_j^{q_{1,1,j}^+}dxdt\right. \\
            & \qquad \qquad \qquad \left.+ \int_0^\delta \int_{A_{\kappa_{n+1}}(t)} (u-\kappa_{n+1})^{q_{1,1,j}^-}\xi_j^{q_{1,1,j}^-} dx dt \right ]\\
            & \leq m^{q_1^+} \sum_{j=1}^m \left [2M_7M_8^n \left(Y_n+Y_n^{q^+} \right) 2^{q_1^+(n+1)} \frac{1}{\kappa^{q_1^-(1-\eta)}} (Y_n+Y_n^{1-\eta}) \right ]\\
            & \leq M_9 M_{10}^n \frac{1}{\kappa^{q_1^- (1-\eta)}} \left
            (Y_n^2+Y_n^{2-\eta}+Y_n^{1+q^+ }+Y_n^{1+q^+ -\eta} \right )
        \end{split}
    \end{align}
    with positive constants $M_9$ and $M_{10}$ depending on the data.


    {\bf (V) Estimating the term $\tilde{Z}_{n+1}$.}
    Similar to step (V) we are going to estimate the term $\tilde{Z}_{n+1}$. First, we have
    \begin{align}\label{V_1}
        \begin{split}
            \tilde{Z}_{n+1} &=\int_0^\delta \int_{\Gamma_{\kappa_{n+1}}} (u-\kappa_{n+1})^{q_2(t,x)}d\sigma dt\\
            & =\int_0^\delta\int_{\Gamma_{\kappa_{n+1}}(t)} (u-\kappa_{n+1})^{q_2(t,x)} \left ( \sum_{j=1}^m \xi_j\right )^{q_2^+}d\sigma dt \\
            & \leq m^{q_2^+} \sum_{j=1}^m \int_0^\delta \int_{\Gamma_{\kappa_{n+1}}(t)} (u-\kappa_{n+1})^{q_2(t,x)} \xi_j^{q_2(t,x)} d\sigma dt \\
            & \leq m^{q_2^+} \sum_{j=1}^m \left [\int_0^\delta \int_{\Gamma_{\kappa_{n+1}}(t)} (u-\kappa_{n+1})^{q_{2,1,j}^+}\xi_j^{q_{2,1,j}^+}d \sigma dt \right.\\
            & \qquad \qquad \qquad \left.+ \int_0^\delta \int_{\Gamma_{\kappa_{n+1}}(t)} (u-\kappa_{n+1})^{q_{2,1,j}^-}\xi_j^{q_{2,1,j}^-} d\sigma dt \right ],
        \end{split}
    \end{align}
    with $q_{2,1,j}^-=\min_{(t,x) \in J_1 \times (\overline{B_j}\cap \Gamma)}q_2(t,x)$. Recall that $p^-_{1,j} \leq q^-_{2,1,j} \leq q^+_{2,1,j}<(p^-_{1,i})_*$ for $j=1,\ldots,m$.

    Then, we fix an index $j\in \{1,\ldots,m\}$ and assume that $r\in \{q_{2,1,j}^-,q_{2,1,j}^+\}$ meaning that  $p^-_{1,j} \leq r <(p^-_{1,j})_*$ and $r \leq q^+$.
    Defining a number $s=s_{1,j}(r)$ through
    \begin{align*}
	s_*=\frac{r+(p_{1,j}^-)_*}{2},
    \end{align*}
    we have that $s<p^-_{1,j} \leq r<s_*<(p^-_{1,j})_*$. Taking into account Proposition \ref{energy_estimate}(2) and twice H\"{o}lder's inequality we obtain
    \begin{align}\label{V_2}
        \begin{split}
            & \int_0^\delta \int_{\Gamma} ((u-\kappa_{n+1})_+ \xi_j)^{r}d\sigma dt \\
            & \leq \hat{C}^{q^+} \left (\int_0^\delta \into |\nabla [(u-\kappa_{n+1})_+\xi_j]|^{s}dxdt\right.\\
            & \qquad \qquad \left. +\int_0^\delta \into|(u-\kappa_{n+1})_+\xi_j|^{s} dx dt\right )\\
            & \qquad \qquad \times \left (\esssup_{0<t<\delta} \int_\Omega (u-\kappa_{n+1})_+^2 dx \right )^{\frac{s-1}{N}} \\
            & \leq \hat{C}^{q^+} \left [ \left ( \int_0^\delta \into |\nabla [(u-\kappa_{n+1})_+\xi_j]|^{p^-_{1,j}}dxdt\right)^{\frac{s}{p^-_{1,j}}} \right.\\
            & \qquad \qquad  \times \left ( \int_0^\delta|A_{\kappa_{n+1}}(t)|dt \right)^{1-\frac{s}{p^-_{1,j}}} \\
            & \qquad \qquad + \left(\int_0^\delta \into |(u-\kappa_{n+1})_+\xi_j|^{p^-_{1,j}}dxdt \right)^{\frac{s}{p^-_{1,j}}}\\
            & \qquad \qquad \left. \times \left ( \int_0^\delta|A_{\kappa_{n+1}}(t)|dt \right)^{1-\frac{s}{p^-_{1,j}}} \right]\\
            & \qquad \qquad  \times \left (\esssup_{0<t<\delta} \int_\Omega (u-\kappa_{n+1})_+^2 dx \right )^{\frac{s-1}{N}},
        \end{split}
    \end{align}
    where $\hat{C}=\max(1,C_\Gamma(p^-_{1,1},N), \ldots, C_\Gamma(p^-_{1,m},N))$ with $C_\Gamma(p^-_{1,j},N)$ being the constant of the energy estimate given in Proposition
    \ref{energy_estimate}(2) for $j=1,\ldots,m$ ensuring that $\hat{C}$ is independent of $j$.
    The right-hand side of (\ref{V_2}) can be estimated through
    \begin{align}\label{V_3}
        \begin{split}
            & \int_0^\delta \int_{\Gamma} ((u-\kappa_{n+1})_+ \xi_j)^{r}d\sigma dt \\
            &\leq M_{11} \left ( \int_0^\delta \into |\nabla (u-\kappa_{n+1})_+|^{p^-_{1,j}}\xi_j^{p^-_{1,j}}dxdt +\int_0^\delta \into u^{q_1(t,x)}dxdt \right.\\
            & \qquad \qquad \left. + \esssup_{0<t<\delta} \int_\Omega (u-\kappa_{n+1})_+^2 dx \vphantom{\int_0^\delta} \right )^{\frac{s}{p^-_{1,j}}+\frac{s-1}{N}}\\
            & \qquad \qquad \times \left ( \int_0^\delta|A_{\kappa_{n+1}}(t)|dt \right)^{1-\frac{s}{p^-_{1,j}}}
        \end{split}
    \end{align}
    with $M_{11}=M_{11}(p^+,q^+,\hat{C},L)$. Applying (\ref{III_3}), (\ref{II_1}), (\ref{I_1}) and \eqref{II_4} to the right-hand side of (\ref{V_3}) yields
    \begin{align}\label{V_3b}
        \begin{split}
            & \int_0^\delta \int_{\Gamma} ((u-\kappa_{n+1})_+ \xi_j)^{r}d\sigma dt \\
            & \leq M_{11} \left ( M_5 M_4^n (Z_n + \tilde{Z}_n)+2^{q_1^+(n+2)}Z_n \right)^{\frac{s}{p^-_{1,j}}+\frac{s-1}{N}}\\
            & \qquad \qquad \times \left ( \frac{2^{q_1^+(n+2)}}{\kappa^{q_1^-}} Z_n\right)^{1-\frac{s}{p^-_{1,j}}}\\
            & \leq M_{12}M_{13}^n (Y_n+Y_n^{2q^+})\left ( \frac{2^{q_1^+(n+2)}}{\kappa^{q_1^-}} Z_n\right)^{1-\frac{s}{p^-_{1,j}}},
        \end{split}
    \end{align}
    where
    \begin{align*}
	\frac{s}{p^-_{1,j}}+\frac{s-1}{N} \leq 2 q^+.
    \end{align*}
    Now, putting $\tilde{\eta}=\max\left(\frac{s_{1,1}(q^+_{2,1,1})}{p^-_{1,1}}, \ldots, \frac{s_{1,m}(q^+_{2,1,m})}{p^-_{1,m}} \right)$ we obtain for the last term in \eqref{V_3b}
    \begin{align}\label{V_4}
        \begin{split}
            \left(\frac{2^{q_1^+(n+2)}}{\kappa^{q_1^-}} Z_n \right)^{1-\frac{s}{p^-_{1,j}}}
            & \leq 2^{q_1^+(n+2)} \left (\frac{1}{\kappa^{q_1^-}}\right)^{1-\tilde{\eta}} \left(Y_n+Y_n^{1-\tilde{\eta}} \right).
        \end{split}
    \end{align}
    Finally, combining (\ref{V_3b}) and (\ref{V_4}) results in
    \begin{align}\label{V_5}
        \begin{split}
            & \int_0^\delta \int_{\Gamma} ((u-\kappa_{n+1})_+ \xi_j)^{r}d\sigma dt \\
            & \leq M_{12}M_{13}^n (Y_n+Y_n^{2q^+})2^{q_1^+(n+2)} \left (\frac{1}{\kappa^{q_1^-}}\right)^{1-\tilde{\eta}} \left(Y_n+Y_n^{1-\tilde{\eta}} \right)\\
            & \leq M_{14}M_{15}^n \frac{1}{\kappa^{q_1^-(1-\tilde{\eta})}} \left(Y_n^2+Y_n^{2-\tilde{\eta}}+Y_n^{2q^++1}+Y_n^{2q^++1-\tilde{\eta}} \right).
        \end{split}
    \end{align}
    From (\ref{V_1}) and (\ref{V_5}) we conclude for $r \in \{q_{2,1,j}^-,q_{2,1,j}^+\}$
    \begin{align}\label{V_6}
        \begin{split}
            \tilde{Z}_{n+1} \leq M_{16}M_{15}^n \frac{1}{\kappa^{q_1^-(1-\tilde{\eta})}} \left(Y_n^2+Y_n^{2-\tilde{\eta}}+Y_n^{2q^++1}+Y_n^{2q^++1-\tilde{\eta}} \right).
        \end{split}
    \end{align}


  {\bf (VI) The iterative inequality for $Y_{n}$.}
    Since $Y_n=Z_n+\tilde{Z}_n$, we derive from (\ref{IV_5})  and (\ref{V_6})
    \begin{align*}
        \begin{split}
            Y_{n+1}
            & \leq K b^n \frac{1}{\kappa^{q_1^- (1-\hat{\eta})}} \left (Y_n^2+Y_n^{2-\eta}+Y_n^{1+q^+ }+Y_n^{1+q^+ -\eta} \right.\\
            & \qquad \qquad \qquad \qquad \left. +Y_n^2+Y_n^{2-\tilde{\eta}}+Y_n^{2q^++1}+Y_n^{2q^++1-\tilde{\eta}}\right )\\
            & \leq 8K b^n \frac{1}{\kappa^{q_1^- (1-\hat{\eta})}}\big(Y_n^{1+\delta_1}+Y_n^{1+\delta_2}\big)
        \end{split}
    \end{align*}
    with $K=\max(M_{9},M_{16}),\,b=\max(M_{10},M_{15})$, $\hat{\eta}=\max(\eta,\tilde{\eta})$ and where $0<\delta_1\le \delta_2$ are given by
    \begin{align*}
        & \delta_1= \min\left(1,1-\eta,q^+,q^+ -\eta,1-\tilde{\eta},2q^+,2q^+-\tilde{\eta}\right),\\
        & \delta_2= \max\left(1,1-\eta,q^+,q^+ -\eta,1-\tilde{\eta},2q^+,2q^+-\tilde{\eta}\right).
    \end{align*}
    We can assume, without loss of generality, that $b>1$. Hence, we may apply Lemma \ref{lemma_geometric_convergence} which ensures that $Y_n \to 0$ as $n\to \infty$ provided
    \begin{align}\label{VI_1}
        \begin{split}
            Y_0
            & =\int_0^\delta \into (u-\kappa)_+^{q_1(t,x)} dxdt +\int_0^\delta \int_{\Gamma} (u-\kappa)_+^{q_2(t,x)} d \sigma dt\\
            & \leq \min \left[\left (\frac{16 K}{\kappa^{q_1^- (1-\hat{\eta})}} \right )^{-\frac{1}{\delta_1}} b^{-\frac{1}{\delta_1^2}}, \left (\frac{16 K}{\kappa^{q_1^- (1-\hat{\eta})}} \right )^{-\frac{1}{\delta_2}}b^{-\frac{1}{\delta_1 \delta_2}-\frac{\delta_2-\delta_1}{\delta_2^2}}\right].
        \end{split}
    \end{align}   
    If we have
    \begin{align}\label{VI_2}
	\begin{split}
	    &\int_0^\delta \into u_+^{q_1(t,x)} dxdt + \int_0^\delta \int_{\Gamma} u_+^{q_2(t,x)} d \sigma dt\\
	    &\leq\min \left[\left (\frac{16 K}{\kappa^{q_1^- (1-\hat{\eta})}} \right )^{-\frac{1}{\delta_1}} b^{-\frac{1}{\delta_1^2}}, \left (\frac{16 K}{\kappa^{q_1^- (1-\hat{\eta})}} \right )^{-\frac{1}{\delta_2}}b^{-\frac{1}{\delta_1 \delta_2}-\frac{\delta_2-\delta_1}{\delta_2^2}}\right],
	\end{split}
    \end{align}
    then (\ref{VI_1}) is obviously satisfied. Thus, choosing $\kappa$ such that
    \begin{align}\label{VI_3}
	\begin{split}
	    \kappa
	    & =\max\left(\max(1,\esssup_{\Omega}u_0), \vphantom{\left[ \int_0^\delta\int_{\Gamma} u_+^{q_2(t,x)} d \sigma dt\right]^{\frac{\delta_1}{q^-_1(1-\hat{\eta})}}} 
	    (16K)^{\frac{1}{q_1^- (1-\hat{\eta})}} b^{\frac{1}{\delta_1 q_1^- (1-\hat{\eta})}+\frac{\delta_2-\delta_1}{\delta_2 q_1^- (1-\hat{\eta})}}
	    \phantom{\int_0^\delta}\right.\\
	    & \qquad\qquad\left.  \times \left (1+\int_0^\delta\into u_+^{q_1(t,x)} dxdt + \int_0^\delta\int_{\Gamma} u_+^{q_2(t,x)} d \sigma dt\right )^{\frac{\delta_2}{q^-_1(1-\hat{\eta})}}\right),
	\end{split}
    \end{align}
    it follows that (\ref{VI_2}) and in particular (\ref{VI_1}) are fulfilled. Since $\kappa_n \to 2\kappa$ as $n \to \infty$ we obtain
    \begin{align*}
        & \esssup_{(0,\delta) \times \Om} u \leq 2\kappa \quad \text{and} \quad \esssup_{(0,\delta) \times \Gamma} u \leq 2\kappa,
    \end{align*}
    where $\kappa$ is defined in (\ref{VI_3}). That means that $u \in L^{\infty}(Q_\delta), L^\infty (\Gamma_\delta)$ with $Q_\delta=(0,\delta)\times \Omega$ as well as $\Gamma_\delta=(0,\delta)\times \Gamma$.


    {\bf (VII) Repeating the iteration.}
    Note that the subsequent constants are independent of $\delta$:
    \begin{align*}
	C_1:=\esssup_{\Omega}u_0, \ 
	C:=(16K)^{\frac{1}{q_1^- (1-\hat{\eta})}} b^{\frac{1}{\delta_1 q_1^- (1-\hat{\eta})}+\frac{\delta_2-\delta_1}{\delta_2 q_1^- (1-\hat{\eta})}}, \
	\beta:=\frac{\delta_2}{q^-_1(1-\hat{\eta})}.
    \end{align*}
    Thus, step (VI) has shown that
    \begin{align*}
	\begin{split}
	   & \max\left( \esssup_{(0,\delta) \times \Omega } u,\esssup_{(0,\delta) \times \Gamma} u\right)\\
	    & \leq 2 \max\left(C_1,C\left (1+\int_0^\delta\into u_+^{q_1(t,x)} dxdt + \int_0^\delta\int_{\Gamma} u_+^{q_2(t,x)} d \sigma dt\right )^{\beta}\right)\\
	    & \leq 2 \max\left(C_1,C\left (1+\int_0^T\into u_+^{q_1(t,x)} dxdt + \int_0^T\int_{\Gamma} u_+^{q_2(t,x)} d \sigma dt\right )^{\beta}\right)\\
	    & =:\tilde{\kappa}_1,
	\end{split}
    \end{align*}
    where $\tilde{\kappa}_1$ is independent of $\delta$. Now we may proceed as in (II)--(VI) replacing $\delta$ by $2\delta$ and starting with $\kappa \geq \tilde{\kappa}_1$. Then, the same calculations as above ensure an estimate of the form
    \begin{align*}
	\begin{split}
	    & \max\left( \esssup_{(0,2\delta) \times \Omega } u,\esssup_{(0,2\delta) \times \Gamma} u\right)\\
	    & \leq 2 \max\left(\tilde{\kappa}_1,C\left (1+\int_0^{2\delta}\into u_+^{q_1(t,x)} dxdt + \int_0^{2\delta}\int_{\Gamma} u_+^{q_2(t,x)} d \sigma dt\right )^{\beta}\right)\\
	    & \leq 2 \max\left(\tilde{\kappa}_1,C\left (1+\int_0^T\into u_+^{q_1(t,x)} dxdt + \int_0^T\int_{\Gamma} u_+^{q_2(t,x)} d \sigma dt\right )^{\beta}\right)\\
	    & =2 \tilde{\kappa}_1 =:\tilde{\kappa}_2.
	\end{split}
    \end{align*}
    Recalling $[0,T]=\bigcup_{i=1}^l [\delta(i-1),\delta i]$ and following this pattern gives the global upper bound
    \begin{align*}
	\begin{split}
	    & \max\left( \esssup_{(0,T) \times \Omega } u,\esssup_{(0,T) \times \Gamma} u\right) \leq \tilde{\kappa}_l=2\tilde{\kappa}_{l-1}= \ldots = 2^{l-1} \tilde{\kappa}_1
	\end{split}
    \end{align*}
    meaning that
    \begin{align*}
	\begin{split}
	    & \max\left( \esssup_{(0,T) \times \Omega } u,\esssup_{(0,T) \times \Gamma} u\right)\\
	    & \leq 2^l \max\left(C_1,C\left (1+\int_0^T\into u_+^{q_1(t,x)} dxdt + \int_0^T\int_{\Gamma} u_+^{q_2(t,x)} d \sigma dt\right )^{\beta}\right).
	\end{split}
    \end{align*}
    This proves the first assertion of the theorem.

    In order to verify the global lower bound for a supersolution, we may argue similarly replacing $u$ by $-u$, $A_\kappa(t)$ by $\tilde{A}_\kappa(t)$ and $\Gamma_\kappa(t)$ by $\tilde{\Gamma}_\kappa(t)$. Additionally, instead of Proposition \ref{proposition_energy_subsolution}, we have to use Proposition \ref{proposition_energy_supersolution}. That finishes the proof of the theorem.
\end{proof}

\end{document}